\documentclass[smallextended,envcountsame]{svjour3}       
\smartqed  
\usepackage[margin=1.2in]{geometry} 

\usepackage[small,bf,hang]{caption} 

\usepackage{graphicx}
\usepackage{subfig}                      
\usepackage{amsmath, amssymb}
\usepackage{hyperref}                    

\usepackage{color}

 \journalname{Journal of Mathematical Chemistry}

\newtheorem{observation}[theorem]{Observation}

\graphicspath{{figures/}}

\begin{document}

\title{Recursive generation of IPR fullerenes
\thanks{Jan Goedgebeur is supported by a Postdoctoral Fellowship of the Research Foundation Flanders (FWO).  Brendan McKay is supported by the Australian
Research Council.}
}


\author{Jan Goedgebeur         \and
        Brendan D.\ McKay 
}


\institute{Jan Goedgebeur \at
              Department of Applied Mathematics, Computer Science \& Statistics, Ghent University, Krijgslaan 281-S9, 9000 Ghent, Belgium\\
              \email{jan.goedgebeur@ugent.be}           
           \and
           Brendan D.\ McKay \at
              Research School of Computer Science, Australian National University, Canberra,  ACT 2601, Australia\\
          	  \email{bdm@cs.anu.edu.au}   
}

\date{}

\maketitle

\vspace{-10mm}

\begin{abstract}
We describe a new construction algorithm for the recursive generation of all non-isomorphic IPR fullerenes. 

Unlike previous algorithms, the new algorithm stays entirely within the class of IPR fullerenes, that is: every IPR fullerene is constructed by expanding a smaller IPR fullerene unless it belongs to a limited class of \textit{irreducible} IPR fullerenes that can easily be made separately.
The class of irreducible IPR fullerenes consists of 36 fullerenes with up to 112 vertices and 4 infinite families of nanotube fullerenes. Our implementation of this algorithm is faster than other generators for IPR fullerenes and we used it to compute all IPR fullerenes up to 400 vertices. \\

Note: this is a preprint which was submitted and subsequently accepted for publication in \textit{Journal of Mathematical Chemistry}. The final publication is available at Springer via \url{http://dx.doi.org/10.1007/s10910-015-0513-7}

\end{abstract}

\section{Introduction}
A \textit{fullerene} is a cubic plane graph where all faces are pentagons or hexagons. Euler's formula implies that a
fullerene with $n$ vertices contains exactly 12 pentagons and $n/2 - 10$ hexagons.

The
\textit{dual} of a fullerene is the plane graph
obtained by exchanging the roles of vertices and faces: the vertex set of the dual graph
is the set of faces of the original graph and two 
vertices in the dual graph are adjacent if and only if the two faces share an edge
in the original graph. The rotational order around the vertices in
the embedding
of the dual fullerene  follows the rotational order of the faces. 

The dual
of a fullerene with $n$ vertices is a \textit{triangulation} (i.e.\ a plane graph where every face is a triangle) which contains 12 vertices with degree 5 and $n/2
- 10$ vertices with degree 6. 

In this article we will mostly use the dual representation of a fullerene, which we call a \textit{dual fullerene}, as this was the most convenient representation for our proofs and implementations.

The discovery in 1985 of the first fullerene molecule, the $C_{60}$ ``buckyball'', won the Nobel Prize for three of its discoverers~\cite{kroto_85}.  Since then many algorithms have been developed to exhaustively list (mathematical models of) fullerene isomers.

The first approach was the spiral algorithm of
Manolopoulos et al.\ in 1991~\cite{manolopoulos_91}. The spiral algorithm
was relatively inefficient and also
incomplete in the sense that not every fullerene isomer could be
generated with it. It was later modified to make it complete, but the
resulting algorithm was not efficient~\cite{manolopoulos_92}. 

An algorithm using folding nets was proposed by Yoshida and Osawa~\cite{yoshida_95} in 1995,
but its completeness remains a difficult open problem.
Liu et al.~\cite{LKSS91} and Sah~\cite{Sah_93} give other algorithms,
but they are also of limited efficiency.

The first complete and efficient generator for fullerenes was developed by Brinkmann and Dress~\cite{brinkmann_97} in 1998 and is called \textit{fullgen}. This algorithm stiches patches together which are bounded by zigzag paths.

In 2012 Brinkmann, Goedgebeur and McKay~\cite{fuller-paper} developed a new generator for all fullerenes called \textit{buckygen} using infinite families of patch replacement operations~\cite{hasheminezhad_08}. \textit{Buckygen} was significantly faster than \textit{fullgen} and contradictory results with \textit{fullgen} led to the detection of a non-algorithmic programming error in \textit{fullgen}. Due to this error some fullerenes were missed starting from 136 vertices. In the meantime this bug has already been fixed and now the results of both generators are in complete agreement. The generator of Brinkmann, Goedgebeur and McKay was also used to prove that the smallest counterexample to the spiral conjecture has 380 vertices~\cite{min_380}.

In this article we define a new construction algorithm for the recursive generation of all non-isomorphic Isolated Pentagon Rule (IPR) fullerenes based on the patch replacement operations of Hasheminezhad, Fleischner and McKay~\cite{hasheminezhad_08}. IPR fullerenes are fullerenes where no two
pentagons share an edge. These fullerenes are especially interesting as they tend to be chemically more stable and thus they are more likely to occur
in nature~\cite{IPR_ref2,IPR_ref}. 

The \textit{face-distance} between two pentagons is the distance between the corresponding vertices of degree 5 in the dual graph. So in IPR fullerenes the minimum face-distance between any two pentagons is at least two. In~\cite{distant_pentagons} we determined a formula for the number of vertices of the smallest fullerenes with a given minimum \textit{face-distance} between any two pentagons.

In Section~\ref{section:construction_operations} we present the construction operations. In Section~\ref{section:irreducible_IPR_fullerenes} we introduce the concept of a \textit{cluster} and determine the \textit{irreducible clusters}. This allows us to prove that the class of irreducible IPR fullerenes consists of 36 fullerenes with up to 112 vertices and 4 infinite families of nanotube fullerenes. Section~\ref{section:generation_algorithm} describes the generation algorithm and how we make sure that no isomorphic fullerenes are output. 

Finally, in Section~\ref{section:results} we compare our implementation of this recursive generation algorithm to other generators for IPR fullerenes.

\section{Construction operations}
\label{section:construction_operations}

A \textit{patch replacement} is a replacement of a connected fragment of a fullerene with a different fragment having identical boundary.  If the new fragment is larger than the old, we call the operation an  \textit{expansion}, and if the new is smaller than the old, we call it a \textit{reduction}.

Since the boundary determines the number of faces in a patch if it contains
fewer than two pentagons~\cite{BGJ06}, and pentagons in fullerenes can
be arbitrarily far apart, an infinite number of different patch expansions is
required to  construct all fullerenes.

Hasheminezhad et al.~\cite{hasheminezhad_08} used two infinite families of expansions to construct all fullerenes (so also non-IPR fullerenes): $L_i$ and $B_{i,j}$. These expansions are sketched in 
Figure~\ref{fig:fullerene_operations}. The lengths of the paths
between the pentagons may vary and for operation $L_i$ 
the mirror image must also be considered. 
All faces drawn completely in
the figure or labelled $f_k$ or $g_k$ have to be distinct. The
faces labelled $f_k$ or $g_k$ can be either pentagons or
hexagons, but when we refer to {\em the} pentagons of the operation, we always
mean the two faces drawn as pentagons.

\begin{figure}[h!t]
	\centering
	\includegraphics[width=0.7\textwidth]{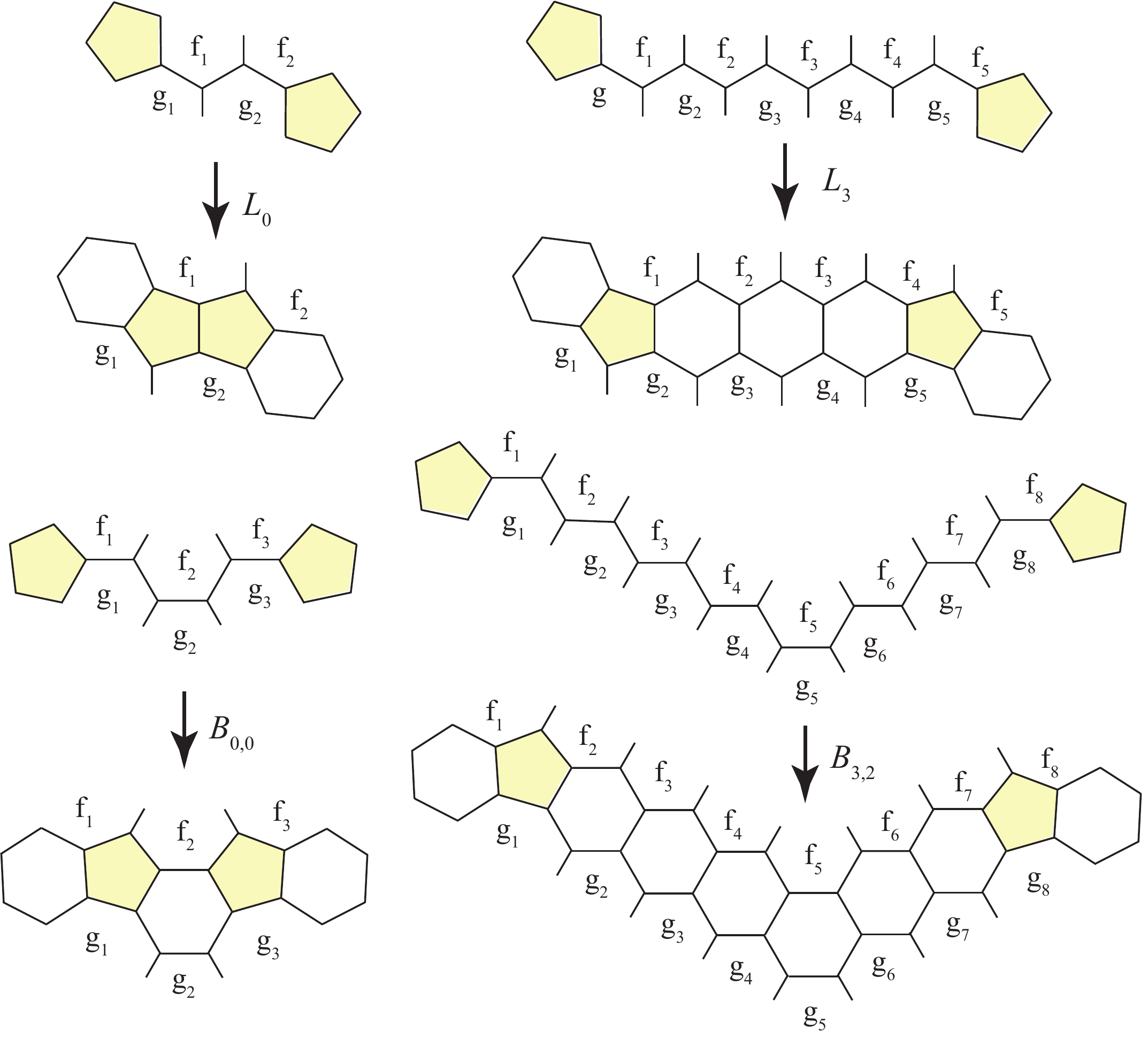}
	\caption{The $L$ and $B$ expansions for fullerenes.}
	\label{fig:fullerene_operations}
\end{figure}

In Figure~\ref{fig:fullerene_operations-dual} the $L$ and $B$
expansions of
Figure~\ref{fig:fullerene_operations} are shown in dual
representation. We will refer to vertices which have degree $k \in
\{5,6\}$ in the dual representation of a fullerene as $k$-vertices. The
solid white vertices in the figure are 5-vertices, the solid black
vertices are 6-vertices and the dashed vertices can be
either. The two 5-vertices which are involved in the reduction and the vertices which must be 6-vertices in the reduction are called the \textit{active} vertices of the reduction. 

\begin{figure}[h!t]
	\centering
	\includegraphics[width=0.8\textwidth]{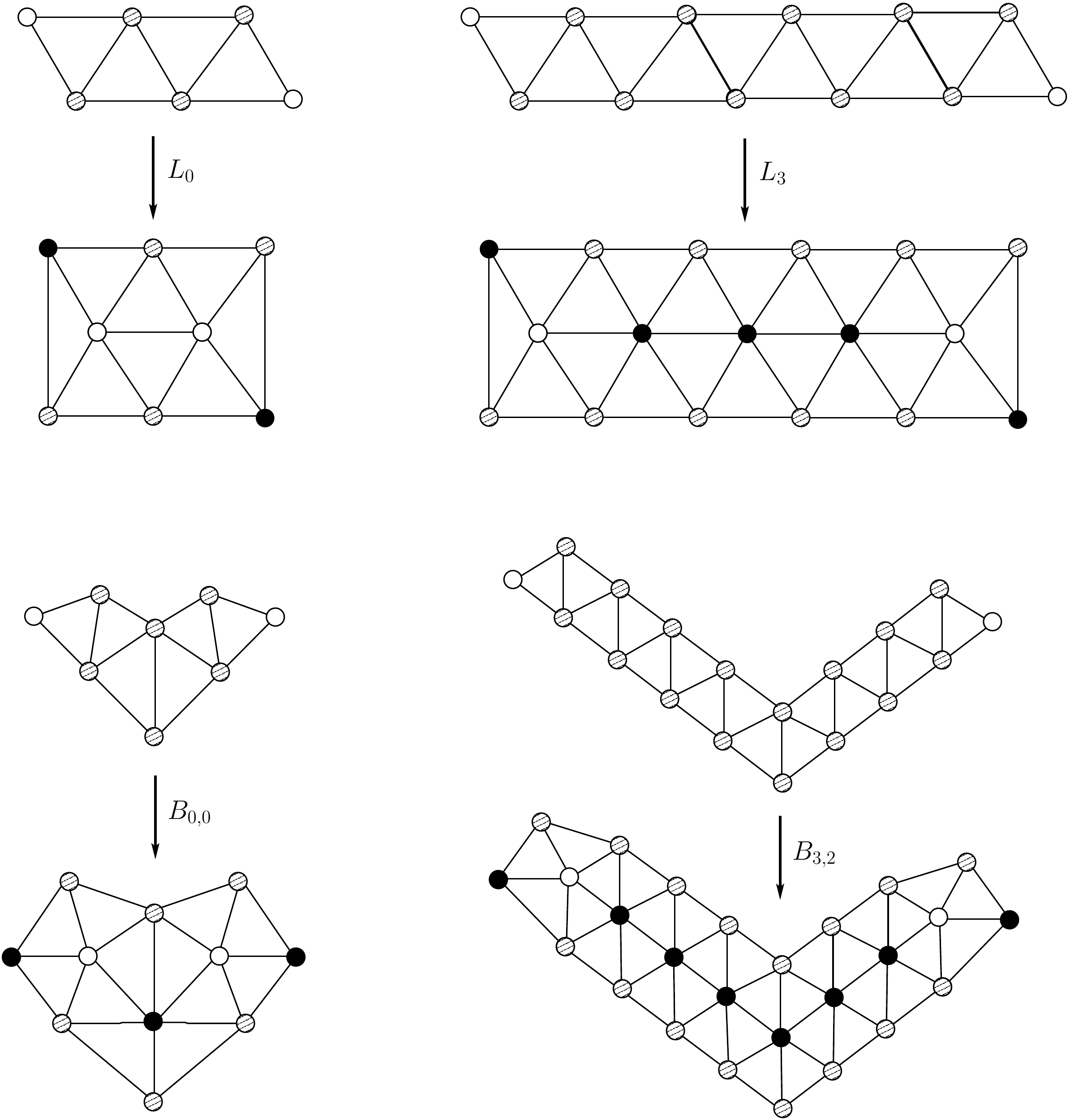}
	\caption{The $L$ and $B$ expansions in dual representation.}
	\label{fig:fullerene_operations-dual}
\end{figure}

Hasheminezhad et al.~\cite{hasheminezhad_08} have proven that every fullerene except $C_{28}(T_d)$ and type-(5,0) nanotube fullerenes, can be reduced to a smaller fullerene by applying an $L$ or $B$ reduction. This means that every fullerene isomer, except $C_{28}(T_d)$ and type-(5,0) nanotube
fullerenes can be constructed by recursively applying expansions of
type $L$ and $B$ to $C_{20}$. 

The program \textit{buckygen}~\cite{fuller-paper} by Brinkmann, Goedgebeur and McKay (which uses the operations of Hasheminezhad et al.), is a generator for all fullerenes, but it also has an option to output only IPR fullerenes by using a filter and some look-aheads. However, many IPR fullerenes are constructed by this generator by applying an expansion to a non-IPR  fullerene. So in order to generate all IPR fullerenes with $n$ vertices, most non-IPR fullerenes with less than $n$ vertices also need to be constructed by the program (see~\cite{fuller-paper} for details).

The construction algorithm which is described in this paper also uses the construction operations of Hasheminezhad et al. and can generate all IPR fullerenes, but stays entirely within the class of IPR fullerenes, that is: IPR fullerenes are constructed from smaller IPR fullerenes. We therefore only apply expansion operations to dual IPR fullerenes which lead to dual IPR fullerenes. We also refer to these operations as \textit{IPR construction operations}. So, for example, we never apply expansions of type $L_0$ or $B_{0,0}$ from Figure~\ref{fig:fullerene_operations-dual} as they result in adjacent 5-vertices.

Figure~\ref{fig:fullerene_operations_ipr} shows some examples of IPR expansions. The solid white vertices are 5-vertices, the solid black vertices are 6-vertices and the dashed ones can be either. If any of the black vertices in the initial patch of the expansion would be a 5-vertex, the expanded dual fullerene would not be IPR. The other IPR expansions are defined similarly.

\begin{figure}[h!t]
	\centering
	\includegraphics[width=0.95\textwidth]{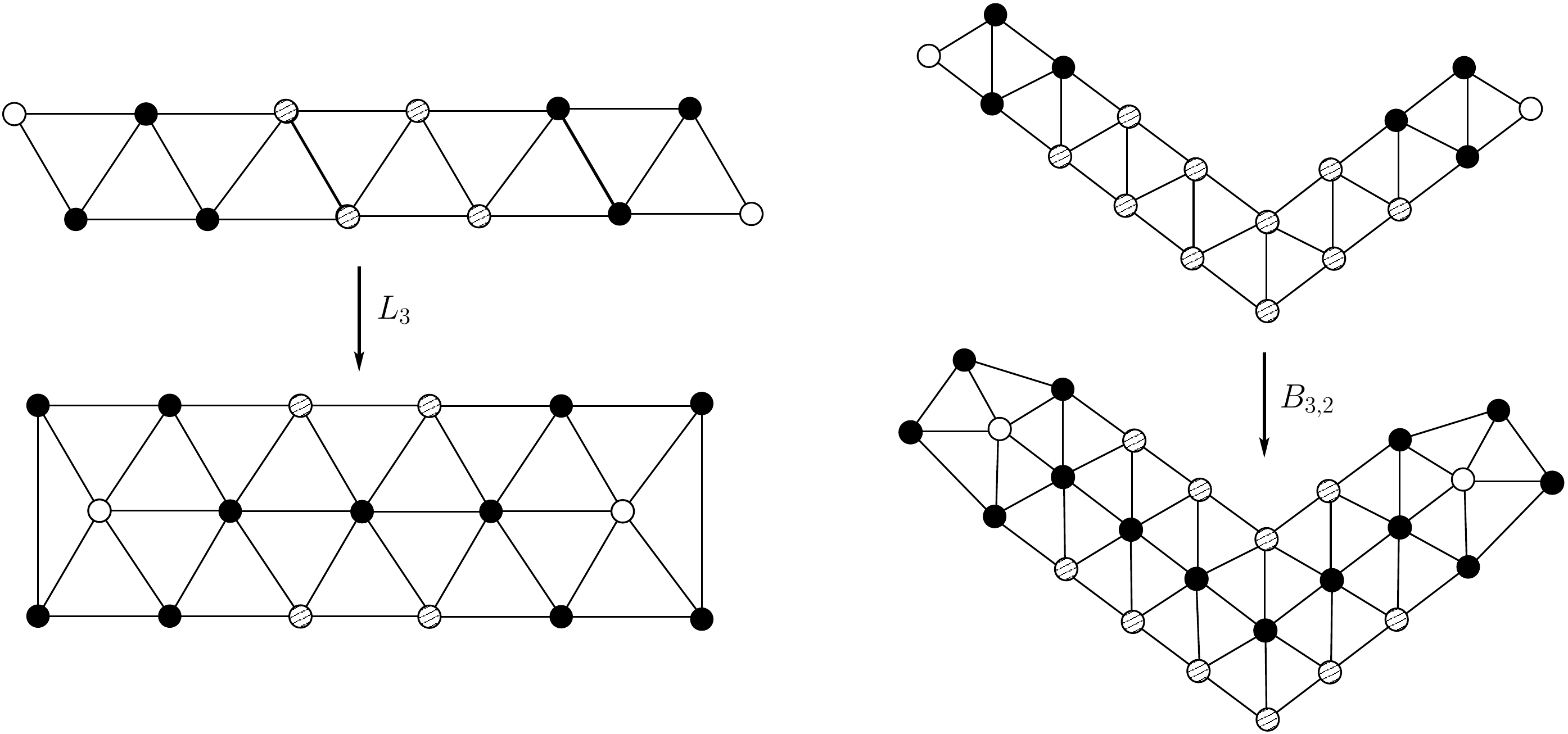}
	\caption{Examples of expansions which can lead to dual IPR fullerenes. }
	\label{fig:fullerene_operations_ipr}
\end{figure}

An IPR fullerene which cannot be reduced to a smaller IPR fullerene by applying one of the reduction operations is called an \textit{irreducible IPR fullerene}. In Section~\ref{section:irreducible_IPR_fullerenes} we prove that the class of irreducible IPR fullerenes consists of 36 fullerenes with up to 112 vertices and 4 infinite families of nanotube fullerenes. 

\section{Irreducible IPR fullerenes}
\label{section:irreducible_IPR_fullerenes}

\subsection{Definitions and preliminaries}

In this section we will classify the irreducible dual IPR fullerenes using the concept of a \textit{cluster}.

A \textit{fullerene patch} is a connected subgraph of a fullerene where all faces except one exterior face are also faces in the fullerene.
Furthermore all boundary vertices have degree 2 or 3 and all non-boundary vertices have degree 3. In the remainder of this article we will abbreviate \textit{``fullerene patch"} as \textit{``patch"}. The \textit{boundary} of a patch is formed by the vertices and edges which are on the unique unbounded face, i.e.\ the \textit{outer face}\index{outer face}.

\begin{definition}[Cluster]\label{def:cluster}
A \textit{$k$-cluster}\index{$k$-cluster} $C$ is a plane graph where all faces except one exterior face are triangles and that has the following properties:

\begin{itemize}
\item All vertices of $C$ have degree at most 6.

\item Vertices which are not on the boundary of $C$ have degree~5 or~6.

\item $C$ contains exactly $k$ vertices with degree 5 which are not on the boundary.

\item No two vertices with degree~5 which are not on the boundary are adjacent.

\item Vertices with degree~5 which are not on the boundary are at distance at least 2 from the boundary.

\item Between any two vertices $a,b$ of $C$ which have degree~5 and which are not on the boundary, there is a path $P$ from $a$ to $b$ so that each edge on $P$ contains exactly one vertex with degree~5 which is not in the boundary.

\item No subgraph of $C$ is a $k$-cluster.

\end{itemize}

\end{definition}

A $k$-cluster for which $k$ is not specified is sometimes just called a \textit{cluster}.  We also assign a colour to the vertices of a cluster: vertices which are on the boundary of the cluster have colour 6 and the colour of the vertices which are not on the boundary is equal to their degree. We also call a vertex with colour 5 a 5-vertex and a vertex with colour 6 a 6-vertex.

We say that a dual fullerene $G$ \textit{contains} a cluster $C$ if and only if $C$ is a subgraph of $G$ and every vertex on the boundary of $C$ has degree 6 in $G$. 

\begin{definition}[Locally reducible cluster]
A cluster is \textit{locally reducible} if there exists an $L$ or $B$-reduction where the active vertices of the reduction are part of the cluster such that the reduced cluster does not contain any adjacent 5-vertices.
\end{definition}

Note that the reduced cluster is not necessarily a cluster. Clusters which are not locally reducible are called \textit{irreducible}.

Lemmas~\ref{obs:reduction_ipr_5_vert} and~\ref{obs:reduction_ipr_6_vert} are useful for the proof of Lemma~\ref{lemma:locally_reducible}.

\begin{lemma} \label{obs:reduction_ipr_5_vert}
Consider a dual fullerene $G$ and a reduction. If $v,w \in V(G)$ are at distance $d$ in $G$ and neither $v$ nor $w$ are active vertices of the reduction, then $v$ and $w$ are at distance at least $d - \bigl\lfloor \frac{d + 1}{3} \bigr\rfloor$ in the reduced dual fullerene.
\end{lemma}

\begin{proof}
%
Let $P$ be a shortest path from $v$ to $w$ after the reduction, and let $d'$ be its length.
$P$ may use the non-boundary edges of the new (smaller) patch, but not more than $\bigl\lceil\frac{d'+1}2\bigr\rceil$ of them, since otherwise two would be adjacent and $P$ could be shortened.  Each such non-boundary edge can be replaced by two edges of the old (larger) patch to form a walk from $v$ to $w$ of length $d\le d'+\bigl\lceil\frac{d'+1}2\bigr\rceil$ before the reduction. This inequality is equivalent to the one required.
\end{proof}


\begin{lemma} \label{obs:reduction_ipr_6_vert}
Consider a dual fullerene $G$ and a reduction. If $v,w \in V(G)$ are at distance $d$ in $G$ and $v$ is a 6-vertex which becomes a 5-vertex after reduction and $w$ is not an active vertex of the reduction, then $v$ and $w$ are at distance at least $d - \bigl\lfloor \frac{d}{3} \bigr\rfloor$ in the reduced dual fullerene.
\end{lemma}

\begin{proof}
The proof is the same as for the previous lemma, noting that $v$ is not incident with a non-boundary edge of the (smaller) patch after the reduction.
%
\end{proof}

%

\begin{lemma} \label{lemma:locally_reducible}
A dual IPR fullerene which contains a locally reducible cluster is reducible to a smaller dual IPR fullerene.
\end{lemma}

\begin{proof}
Consider a dual IPR fullerene $G$ which contains a locally reducible cluster $C$. Let $G'$ be the dual fullerene obtained by applying a reduction from $C$. The only possibility such that $G'$ would not be IPR is that a 5-vertex which is part of $C$ or a 6-vertex of $C$ which becomes a 5-vertex after reduction would be adjacent to a 5-vertex which is not part of the cluster. 

Let $v$ be a 5-vertex of $G$ which is not part of $C$. It follows from Definition~\ref{def:cluster} that 5-vertices which are not part of the cluster, are at distance at least 3 from 5-vertices which are part of the cluster. 

Let $w$ be a 5-vertex which is in $C$ and which is not an active vertex of the reduction. It follows from Lemma~\ref{obs:reduction_ipr_5_vert} that $v$ and $w$ are at distance at least 2 in $G'$.

Now let $w$ be a 6-vertex which becomes a 5-vertex after reduction. Since $w$ is adjacent to a 5-vertex in $C$, it follows from Definition~\ref{def:cluster} that $v$ and $w$ are at distance at least 2 in $G$. Thus it follows from Lemma~\ref{obs:reduction_ipr_6_vert} that $v$ and $w$ are at distance at least 2 in $G'$.

Thus $G'$ does not contain any adjacent 5-vertices.
\end{proof}

Note that if a dual fullerene contains multiple clusters, they are distinct in the sense that for every two clusters in a dual fullerene the set of 5-vertices is disjoint, but they may have some 6-vertices in common.

\subsection{Reducibility of $k$-clusters $(1 \le k \le 6)$}

\begin{lemma} \label{lemma:1_cluster}
All dual IPR fullerenes which contain only 1-clusters are reducible to a smaller dual IPR fullerene.
\end{lemma}

\begin{proof}
In~\cite{hasheminezhad_08} it was proven that in a dual IPR fullerene, at least one shortest path between any two 5-vertices forms a valid $L$ or $B$-reduction (not necessarily to a dual IPR fullerene). Each cluster contains one 5-vertex, thus all vertices at distance at most 2 from each 5-vertex are 6-vertices. 

Consider a dual IPR fullerene $G$ which contains only 1-clusters. Let $G'$ be the graph obtained by applying the shortest reduction between two 5-vertices $a,b \in V(G)$. Let $a'$ (respectively $b'$) be the 6-vertex in $G$ which is adjacent to $a$ (respectively $b$) which is transformed into a 5-vertex by the reduction.
It follows from Lemma~\ref{obs:reduction_ipr_5_vert} that the distance in $G'$ between 5-vertices which were not involved in the reduction is at least 2.
It follows from Lemma~\ref{obs:reduction_ipr_6_vert} that the distance in $G'$ between $a'$ (or $b'$) and a 5-vertex which is not modified by the reduction is at least 2.

Suppose $a$ and $b$ are at distance $d$ in $G$. Note that $d$ is at least 3 since $a$ and $b$ lie in different clusters. Since we performed the shortest reduction between $a$ and $b$, $a'$ and $b'$ are at distance at least $d - 2$ in $G'$. If $d > 3$ there is not a problem. If $d = 3$, $a'$ and $b'$ could be at distance 1 in $G'$. However this would imply that $G'$ has a non-trivial cyclic 5-edge-cut and is thus a type-(5,0) nanotube (see~\cite{thesis_jan} for details). But this is not possible since $G$ is IPR. Thus $G'$ is a dual IPR fullerene.
\end{proof}

Using an algorithm that generates all $k$-clusters for given $k$ (see \cite{thesis_jan} for details), we tested all $k$-clusters for local reducibility. We obtained the following results:

\begin{observation} \label{lemma:235_cluster}
All $k$-clusters with $k \in \{2, 3, 5\}$ are locally reducible.
\end{observation}

Applying Lemma~\ref{lemma:locally_reducible} to Observation~\ref{lemma:235_cluster} gives us the following corollary:

\begin{corollary}
Every dual IPR fullerene which contains a $k$-cluster $(k \in \{2, 3, 5\})$ is reducible to a smaller dual IPR fullerene.
\end{corollary}

\begin{observation}
There is exactly one 4-cluster which is not locally reducible.
\end{observation}

\begin{figure}[h!t]
	\centering
	\includegraphics[width=0.45\textwidth]{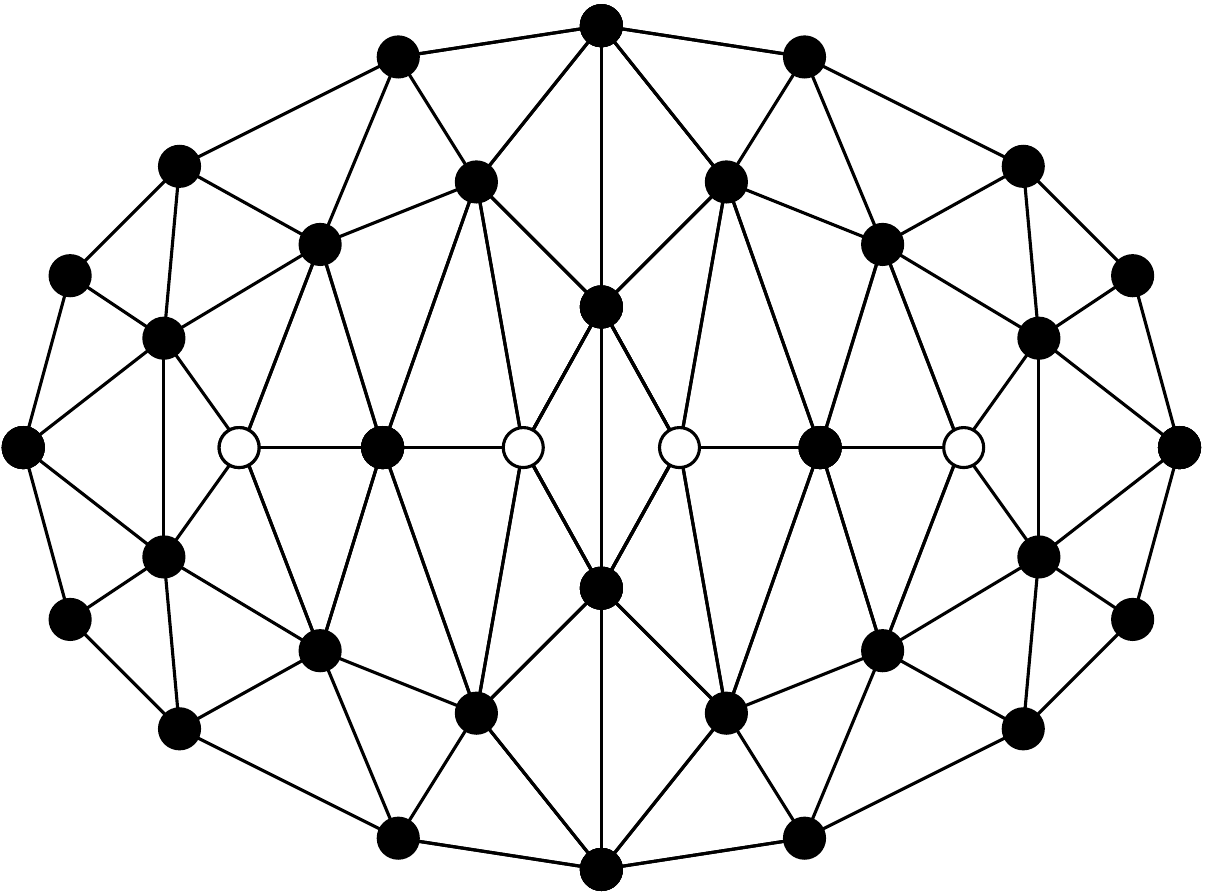}
	\caption{A locally irreducible 4-cluster.}
	\label{fig:irred_4-cluster}
\end{figure}

This cluster is depicted in Figure~\ref{fig:irred_4-cluster}. The four 5-vertices are white and the other vertices are 6-vertices. 
Every dual IPR fullerene which contains this cluster has a $B_{2,2}$-reduction to a smaller dual IPR fullerene
unless the vertex $x$ displayed in Figure~\ref{fig:irred_4-cluster-B_red} is a 5-vertex. The path of vertices which is going to be reduced by the $B_{2,2}$-reduction is drawn with dashed edges (assuming $x$ is not a 5-vertex). In principle $x$ can be a vertex which is part of the cluster, but this is not a problem for the reduction. If $x$ is a 5-vertex, there is an $L_2$-reduction which yields a dual IPR fullerene. This is shown in Figure~\ref{fig:irred_4-cluster-L_red}. The reduced dual fullerene is IPR since $y$ is a 6-vertex, otherwise the dual fullerene before reduction was not IPR. 
In principle $y$ might be identical to one of the vertices which is part of the cluster. This gives us the following corollary:

\begin{corollary} \label{lemma:4_cluster}
Every dual IPR fullerene which contains a 4-cluster is reducible to a smaller dual IPR fullerene.
\end{corollary}

\begin{figure}[h!t]
    \centering
    \subfloat[]{\label{fig:irred_4-cluster-B_red}\includegraphics[width=0.4\textwidth]{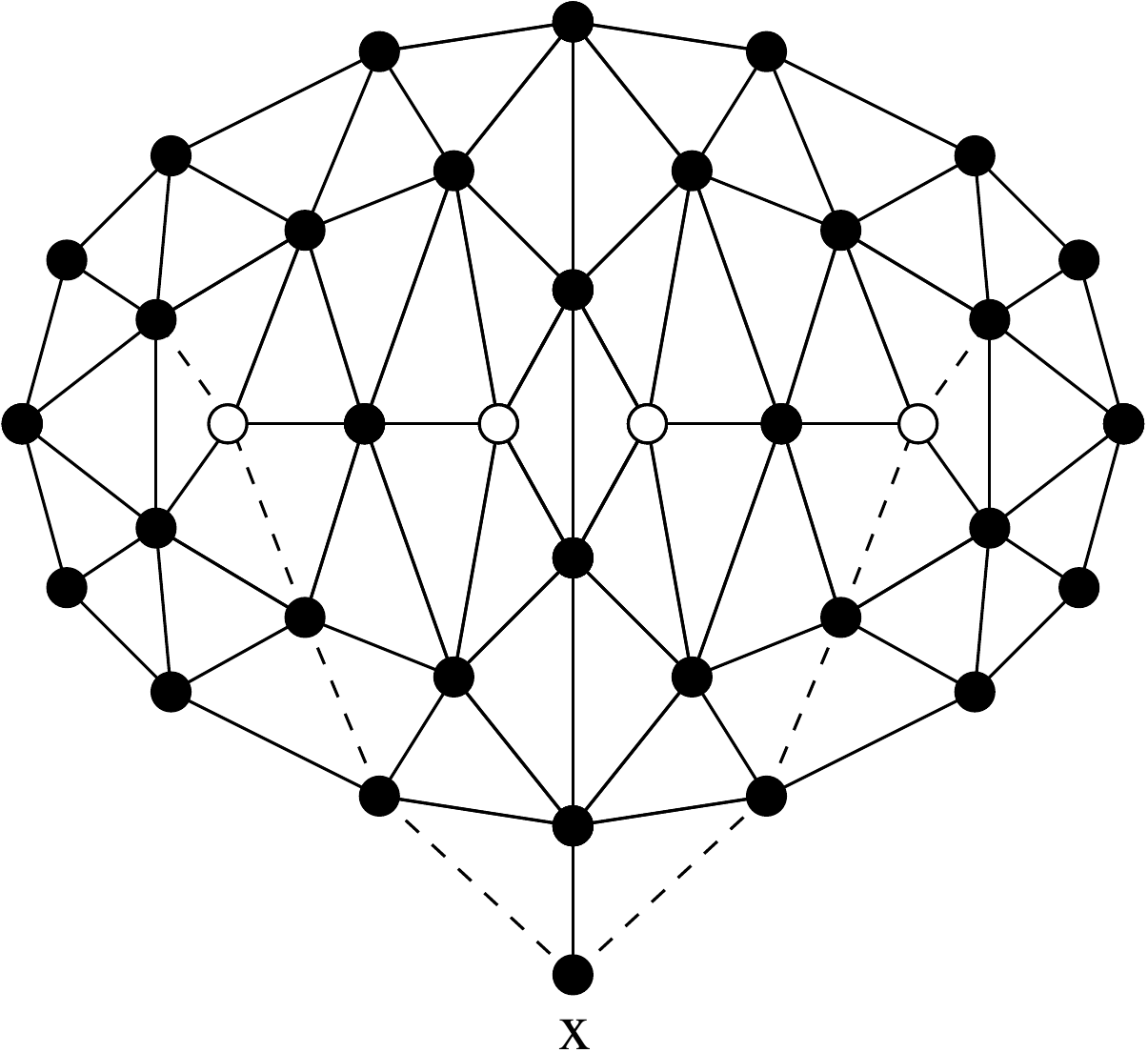}}\qquad\ 
    \subfloat[]{\label{fig:irred_4-cluster-L_red}\includegraphics[width=0.4\textwidth]{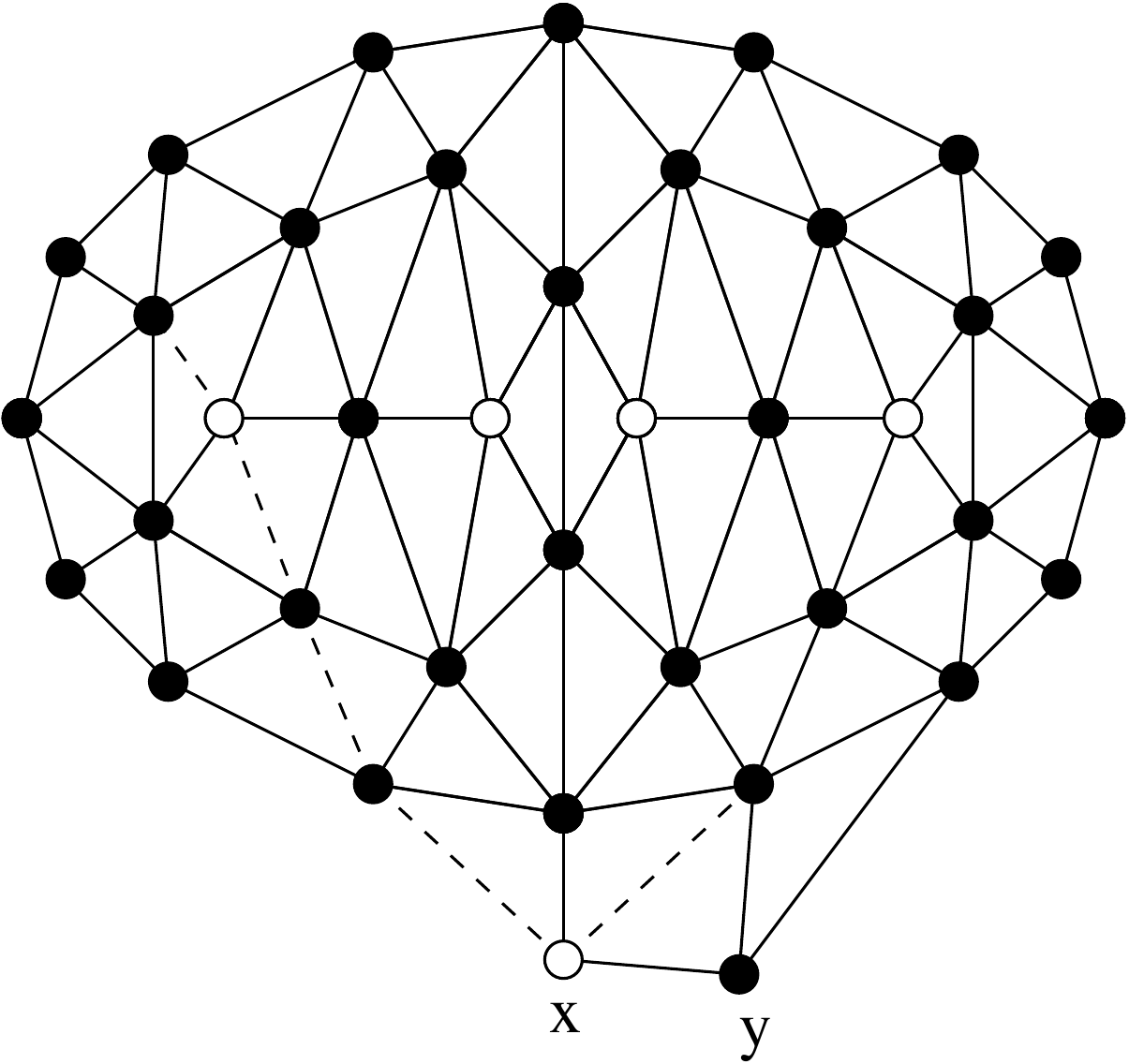}}
    \caption{A locally irreducible 4-cluster which has a $B_{2,2}$-reduction (i.e.~Figure~\ref{fig:irred_4-cluster-B_red}) or an $L_2$-reduction (i.e.~Figure~\ref{fig:irred_4-cluster-L_red}).}
\end{figure}

Using the generator for $k$-clusters we also obtained the following result:

\begin{observation}
There are exactly six 6-clusters which are not locally reducible.
\end{observation}

\begin{figure}[h!t]
	\centering
	\includegraphics[width=0.55\textwidth]{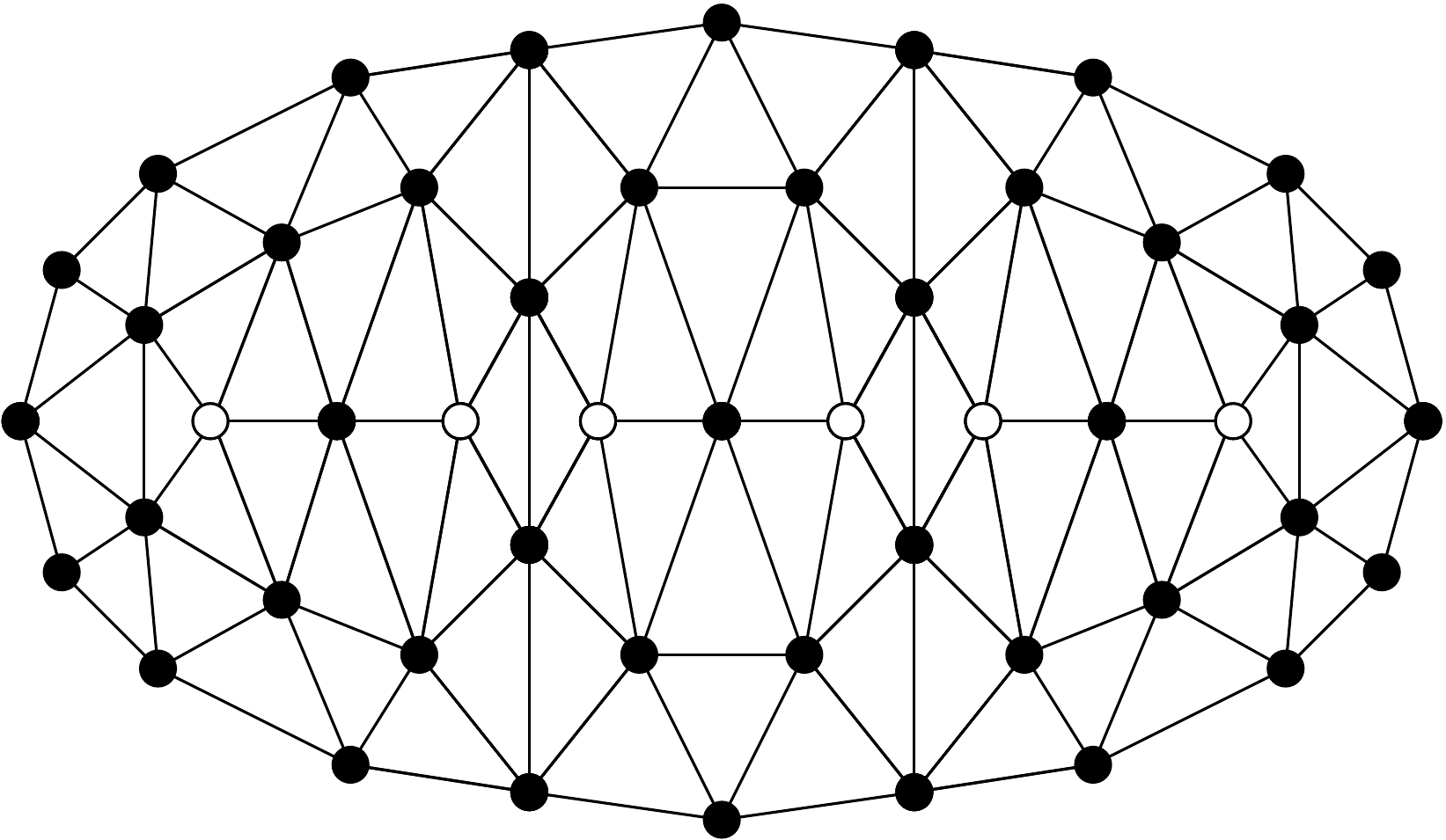}
	\caption{A locally irreducible 6-cluster, called \textit{straight-cluster}.}
	\label{fig:irred_6-cluster-straight}
\end{figure}

The first cluster is depicted in Figure~\ref{fig:irred_6-cluster-straight}. The six 5-vertices are white and the other vertices are 6-vertices. We call this a \textit{straight-cluster}. 
Every dual IPR fullerene which contains this cluster has an $L_6$-reduction to a smaller dual IPR fullerene
unless vertex $a$ or $b$ displayed in Figure~\ref{fig:irred_6-cluster-straight-L6} is a 5-vertex. This is shown in Figure~\ref{fig:irred_6-cluster-straight-L6}. Also here $a$ and $b$ may be part of the cluster. The path of vertices which is going to be reduced by the $L_6$-reduction is drawn with dashed edges. If $a$ or $b$ is a 5-vertex, there is an $L_2$-reduction which yields an IPR fullerene. This is shown in Figure~\ref{fig:irred_6-cluster-straight-L2} where it is assumed that $a$ is a 5-vertex. The reduced dual fullerene is IPR since $b$ is a 6-vertex, otherwise the original dual fullerene was not IPR. This gives us the following corollary:

\begin{corollary}
Every dual IPR fullerene which contains a straight-cluster is reducible to a smaller IPR fullerene.
\end{corollary}

\begin{figure}[h!t]
    \centering
    \subfloat[]{
    \label{fig:irred_6-cluster-straight-L6}
    \includegraphics[width=0.45\textwidth]{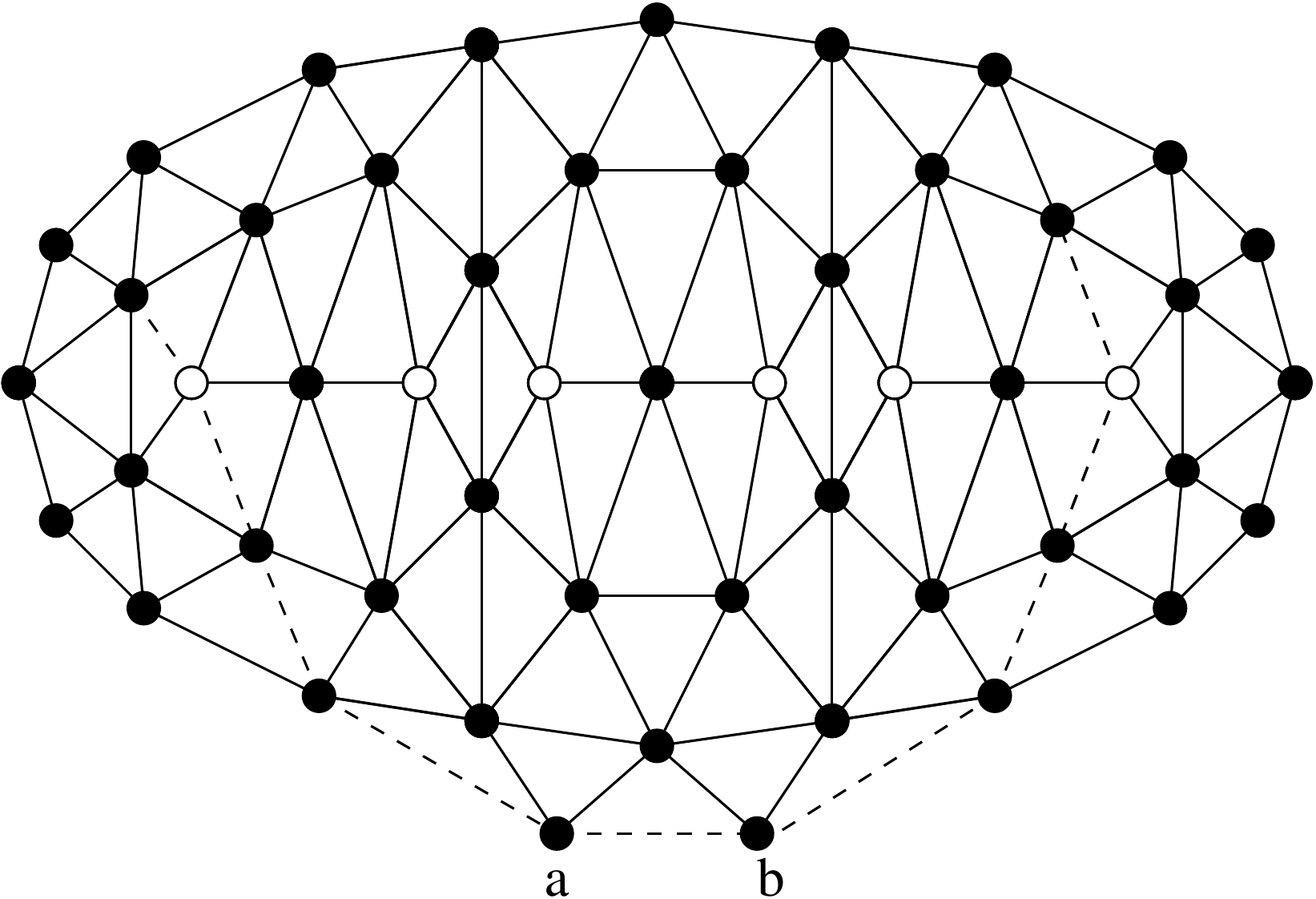}
    }
    \ \ \ \ 
    \subfloat[]{
    \label{fig:irred_6-cluster-straight-L2}
    \includegraphics[width=0.45\textwidth]{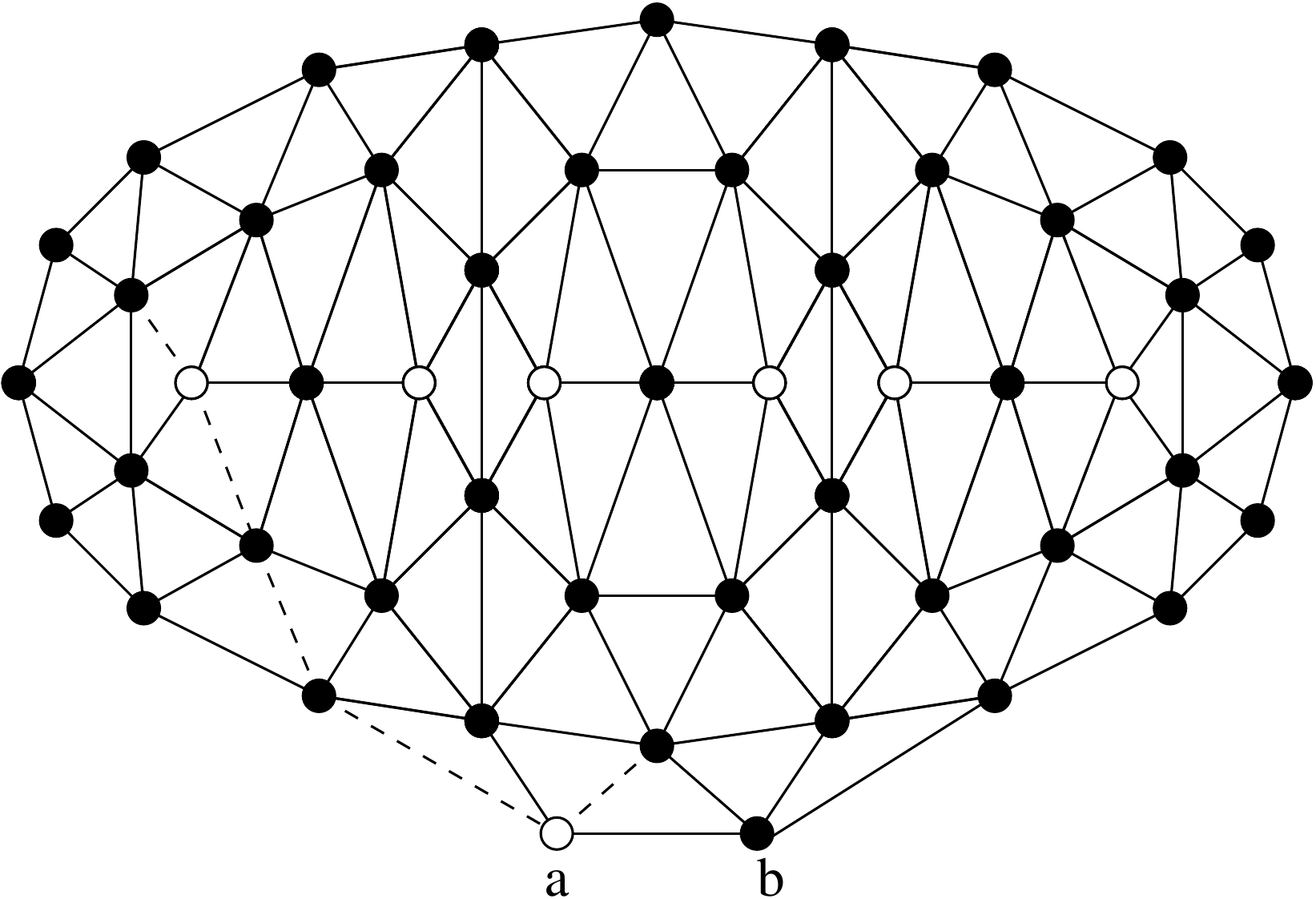}
    }
    \caption{Straight-cluster which has an $L_6$-reduction (i.e.~Figure~\ref{fig:irred_6-cluster-straight-L6}) or an $L_2$-reduction (i.e.~Figure~\ref{fig:irred_6-cluster-straight-L2}).}
\end{figure}

We call the cyclic sequence of the degrees of the vertices in the boundary of a patch in clockwise or counterclockwise order the \textit{boundary sequence}\index{patch!boundary!sequence} of a patch.

A \textit{cap}\index{cap} is a fullerene patch which contains 6 pentagons and has a boundary sequence of the form $(23)^l (32)^m$. Such a boundary is represented by the parameters $(l,m)$. In the literature, the vector $(l,m)$ is also called the \textit{chiral vector} (see~\cite{saito1998physical}). When we speak about \textit{caps} in the remainder of this article, we more specifically mean caps with a boundary sequence of the form $(23)^l (32)^m$. 
Not every patch of 6 pentagons can be completed with hexagons to a patch with a boundary sequence of the form $(23)^l (32)^m$ (see~\cite{justus_03} for an example), but the patches with 6 pentagons which we will discuss in the remainder of this section all can be completed with hexagons to a boundary of the form $(23)^l (32)^m$. 

A cap with boundary parameters $(m, l)$ is the mirror image of a cap with boundary $(l,m)$. A cap has a valid reduction if and only if its mirror image is also reducible. Therefore we will assume that $l \ge m$.
It follows from the results of Brinkmann~\cite{thesis_gunnar} that a (fullerene) patch which contains 6 pentagons and which can be completed with hexagons to a boundary of the form $(23)^l (32)^m$ has unique boundary parameters, i.e.\ it cannot be completed to a boundary with parameters $(l',m')$ where $l'$ is different from $l$ or $m'$ is different from $m$.

The second irreducible 6-cluster is depicted in Figure~\ref{fig:irred_6-cluster-star}. We call this a \textit{distorted star-cluster}. 
By checking all possible reductions, it can be seen that for any dual IPR fullerene which contains this cluster there are no reductions to a smaller dual IPR fullerene where both 5-vertices of the reduction are in the distorted star-cluster.

\begin{figure}[h!t]
	\centering
	\includegraphics[width=0.5\textwidth]{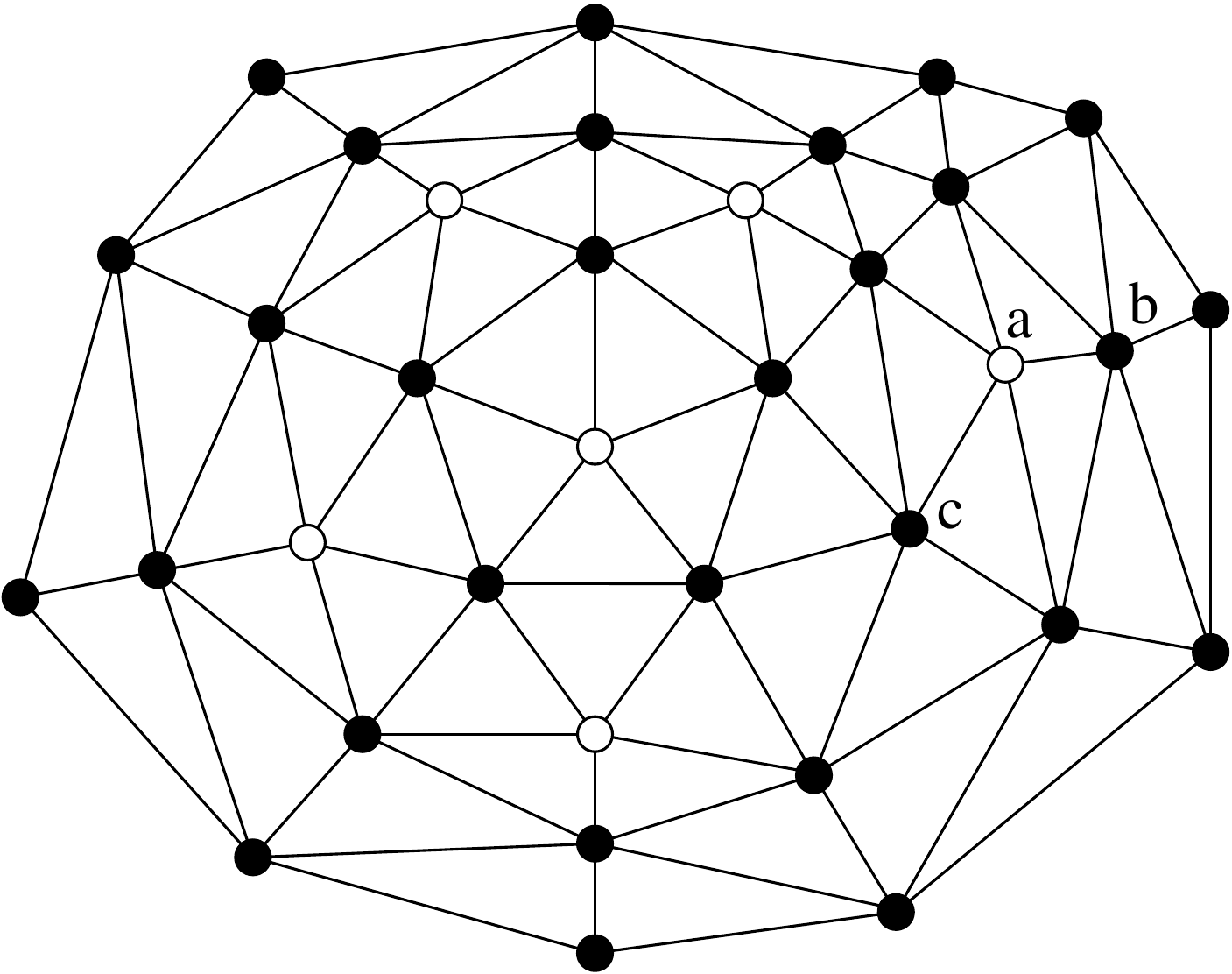}
	\caption{A locally irreducible 6-cluster, called \textit{distorted star-cluster}.}
	\label{fig:irred_6-cluster-star}
\end{figure}

Caps which contain the dual of a distorted star-cluster as a subgraph have boundary parameters (6,5).  Adding a ring of hexagons (or a ring of 6-vertices in the dual) to a cap does not change the boundary parameters of the cap.
Note that there are multiple ways of gluing together two caps with boundary parameters $(l,0)$ to a fullerene.
We call an $(l,m)$ ring of hexagons of an IPR fullerene \textit{removable} if there is a way of removing that ring of hexagons such that the reduced fullerene is still IPR.

We call a cap which contains at least one pentagon in its boundary a \textit{kernel}\index{kernel}. Clearly, every cap has a kernel.


The program from Brinkmann et al.\ described in~\cite{brinkmann_02} generates all nanotube caps which are non-isomorphic as infinite half-tubes. This is done by first generating all non-isomorphic nanotube caps and then filtering out the ones which are non-isomorphic as infinite half-tubes. We modified the program so it outputs all non-isomorphic nanotubes (thus also the ones which are isomorphic as infinite half-tubes). By using this modified version of the generator, we were able to generate all IPR (6,5) kernels. The largest one has 73 vertices, so an IPR fullerene which contains a (6,5) cap and has no removable (6,5) hexagon rings has at most $2 \cdot 73 + 2 \cdot (6+5) = 168$ vertices. The $2 \cdot (6+5)$ represents a ring of hexagons, since the fullerene consisting of 2 IPR kernels may not be IPR.

Using the corrected version of \textit{fullgen}~\cite{brinkmann_97}, we determined all IPR fullerenes up to 168 vertices which have a (6,5) boundary and do not have any removable (6,5) hexagon rings. There are 11 such fullerenes and each of them is reducible to a smaller IPR fullerene. The largest one has 106 vertices. These results have been independently confirmed by \textit{buckygen}~\cite{fuller-paper} using a filter and look-aheads for IPR fullerenes. All of the dual (6,5) caps in these 11 dual IPR fullerenes contain a connected subgraph with six 5-vertices which is isomorphic to a subgraph of the distorted star-cluster.

Consider the directed edge $(a,b)$ from the distorted star-cluster from Figure~\ref{fig:irred_6-cluster-star}. 
If a ring of 6-vertices is added to a dual (6,5) cap which contains $(a,b)$, the straight path starting from $(a,b)$ still exits the cap at the same relative position in the larger dual cap.
Consider a dual IPR fullerene $F$ which has a (6,5) boundary. If there is an $L$ or $B$-reduction which starts from $(a,b)$ and where the second 5-vertex of the reduction is part of the other dual cap of $F$, then the dual fullerene $F'$ obtained by adding a (6,5) ring of 6-vertices to $F$ is still reducible by the same reduction (but which now has one additional 6-vertex). So if the reduction in $F$ was an $L_x$ reduction, it will be an $L_{x+1}$ reduction in $F'$. 
(Note that a reduction where $a$ is one of the 5-vertices involved in the reduction and where $b$ is part of the reduction path can only produce a smaller dual IPR fullerene if vertex $c$ (from Figure~\ref{fig:irred_6-cluster-star}) is the 6-vertex which is transformed into a 5-vertex by the reduction.)

We then added (6,5) rings of 6-vertices to these 11 dual fullerenes which have a (6,5) boundary and do not have any removable (6,5) rings of 6-vertices. When 5 rings of 6-vertices have been added, there is a reduction from $(a,b)$ to the other dual cap in each of the 11 cases. So all dual fullerenes of these 11 types with at least 5 (6,5) rings of 6-vertices are reducible to a smaller dual IPR fullerene. We also verified that each of these 11 types of dual fullerenes with less than 5 rings of 6-vertices are reducible as well.

This gives us the following corollary:

\begin{corollary} \label{cor:65-boundary}
Every dual IPR fullerene which contains a (6,5) boundary is reducible to a smaller dual IPR fullerene.
\end{corollary}

There is a dual (6,5) kernel which is a subgraph of the distorted star-cluster. So if a dual fullerene contains a distorted star-cluster, it also has a dual (6,5) kernel and thus also a (6,5) boundary. This gives us:

\begin{corollary}
Every dual IPR fullerene which contains a distorted star-cluster is reducible to a smaller dual IPR fullerene.
\end{corollary}

The remaining four locally irreducible 6-clusters are depicted in Figure~\ref{fig:4_irred_6_clusters}. We call them clusters I, II, III and IV respectively. Dual caps which contain cluster I, II, III or IV as a subgraph have boundary parameters (5,5), (8,2), (9,0) and (10,0) respectively.

\begin{figure}[h!t]
    \centering
    \subfloat[cluster I]{
    \label{fig:cluster_I}
    \includegraphics[width=0.4\textwidth]{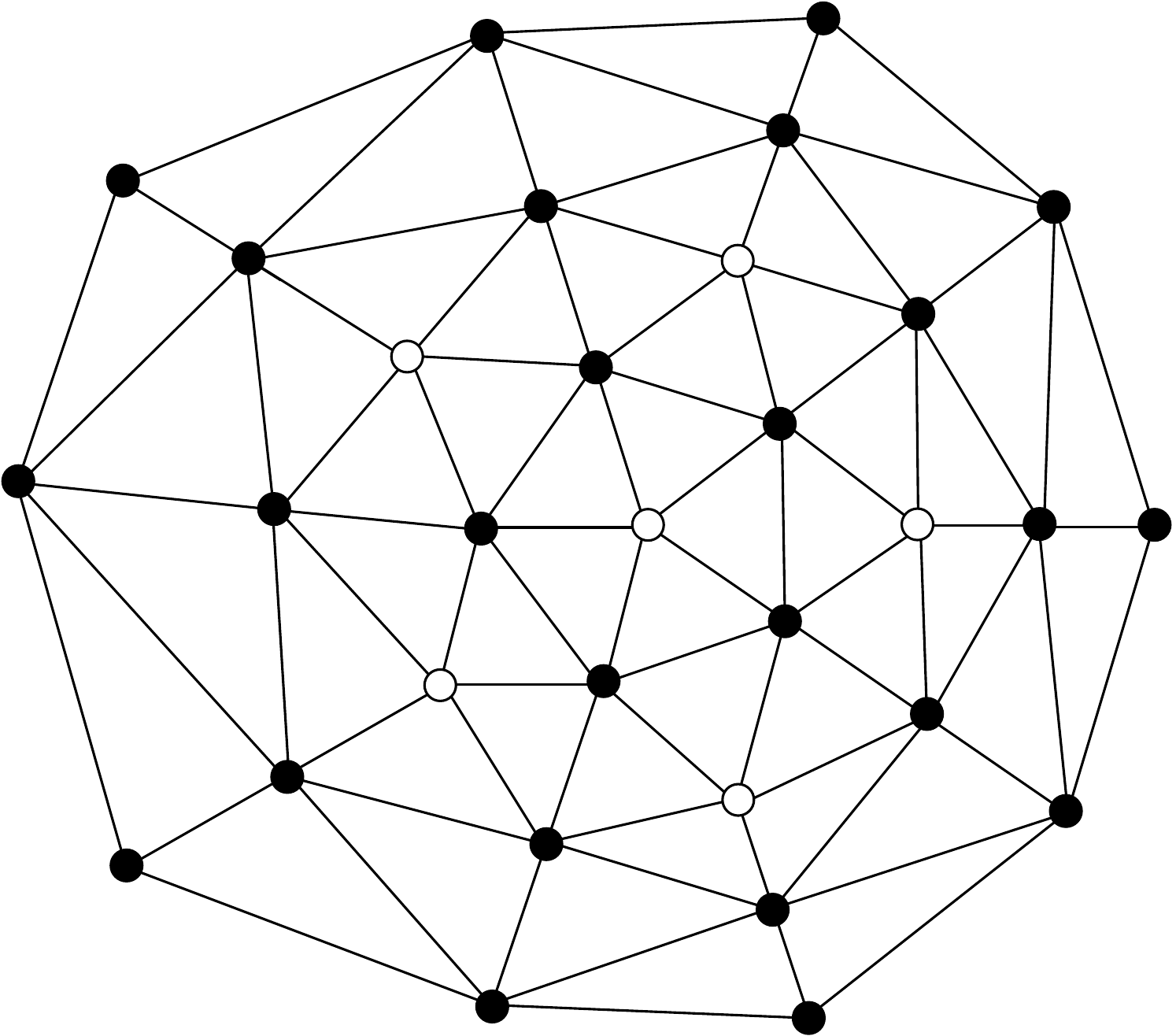}
    }\qquad\ 
    \subfloat[cluster II]{
    \label{fig:cluster_II}    
    \includegraphics[width=0.4\textwidth]{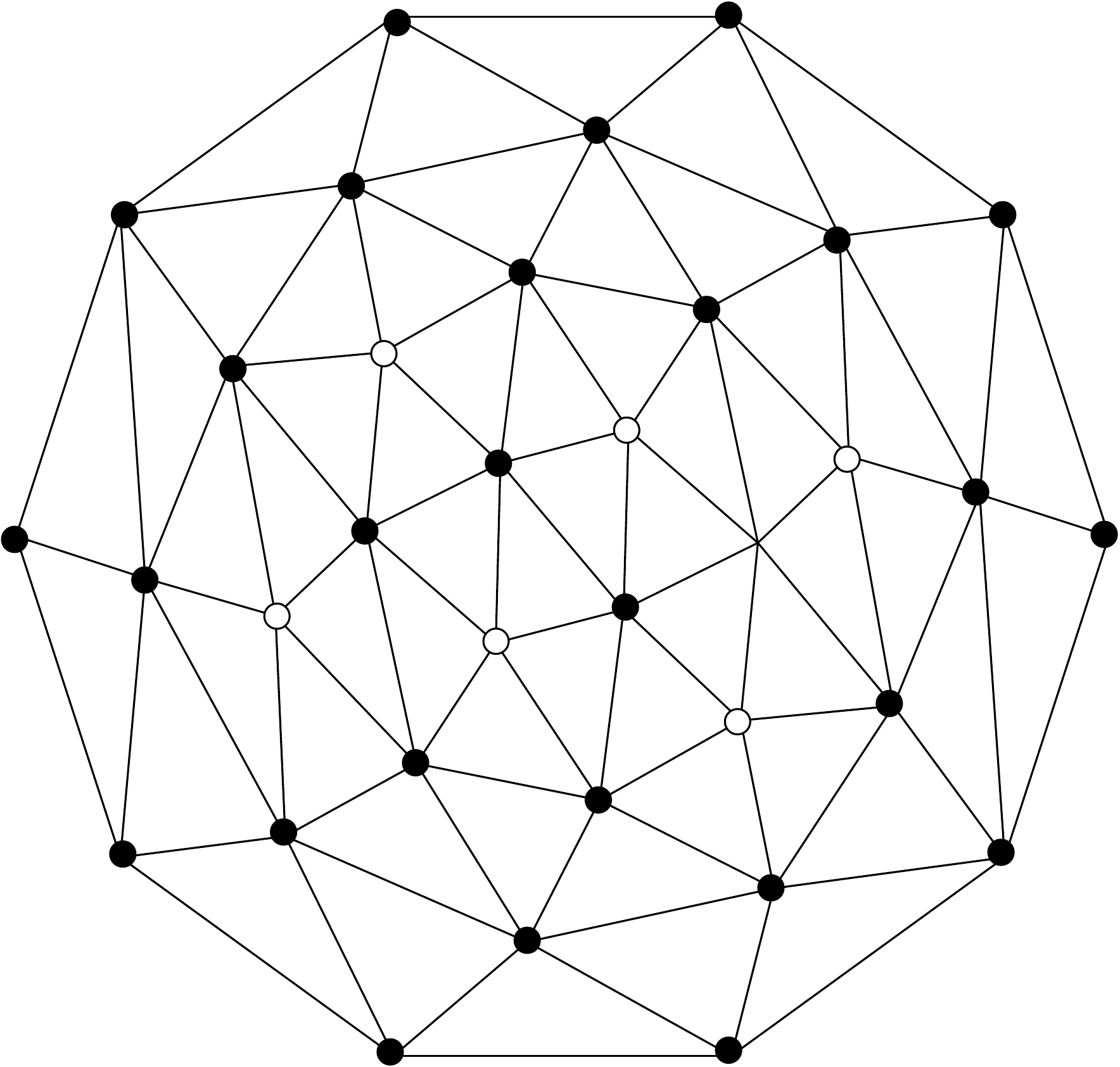}
    }
    
    \subfloat[cluster III]{
    \label{fig:cluster_III}
    \includegraphics[width=0.4\textwidth]{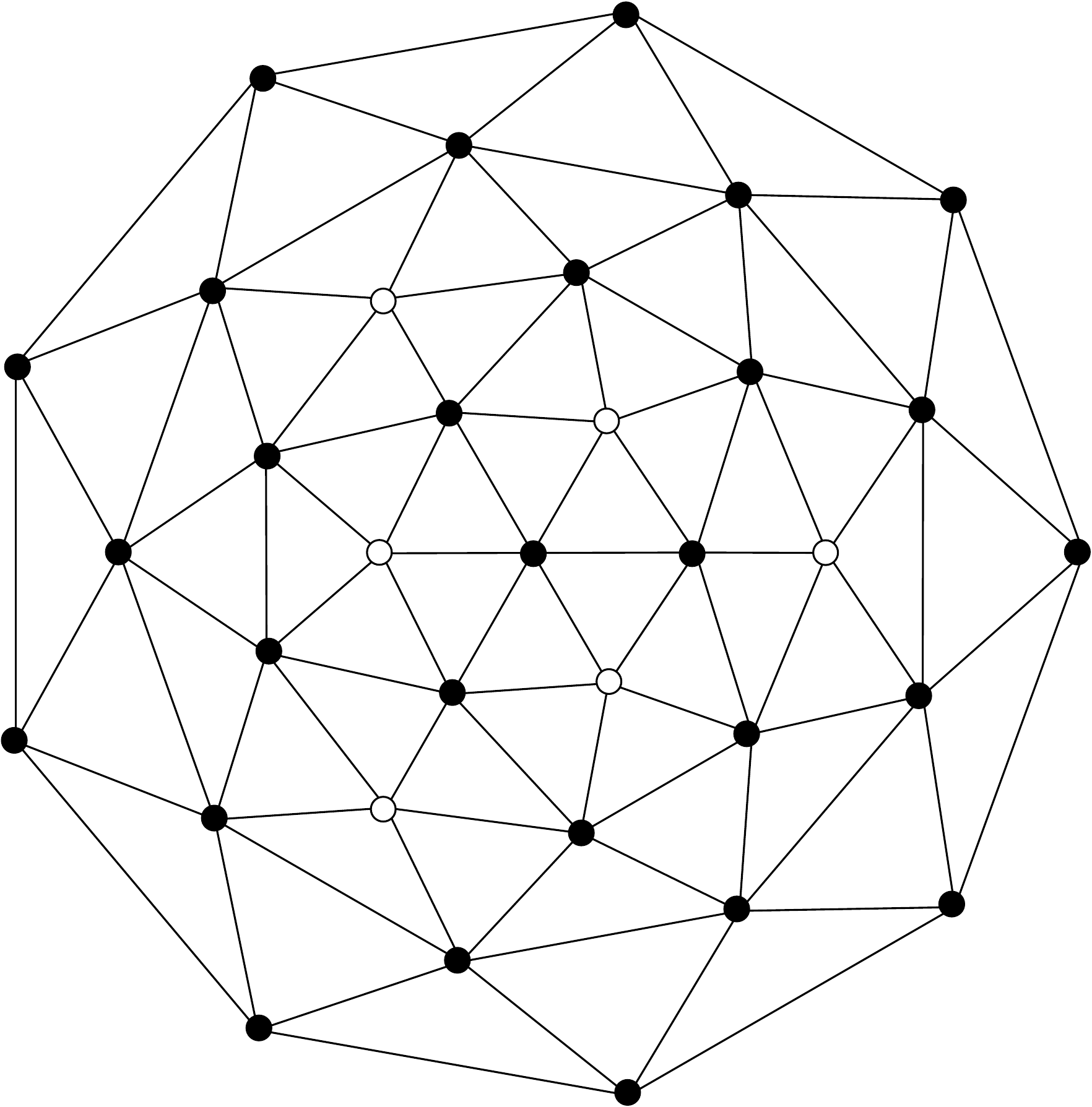}
    }\qquad\ 
    \subfloat[cluster IV]{
    \label{fig:cluster_IV}    
    \includegraphics[width=0.4\textwidth]{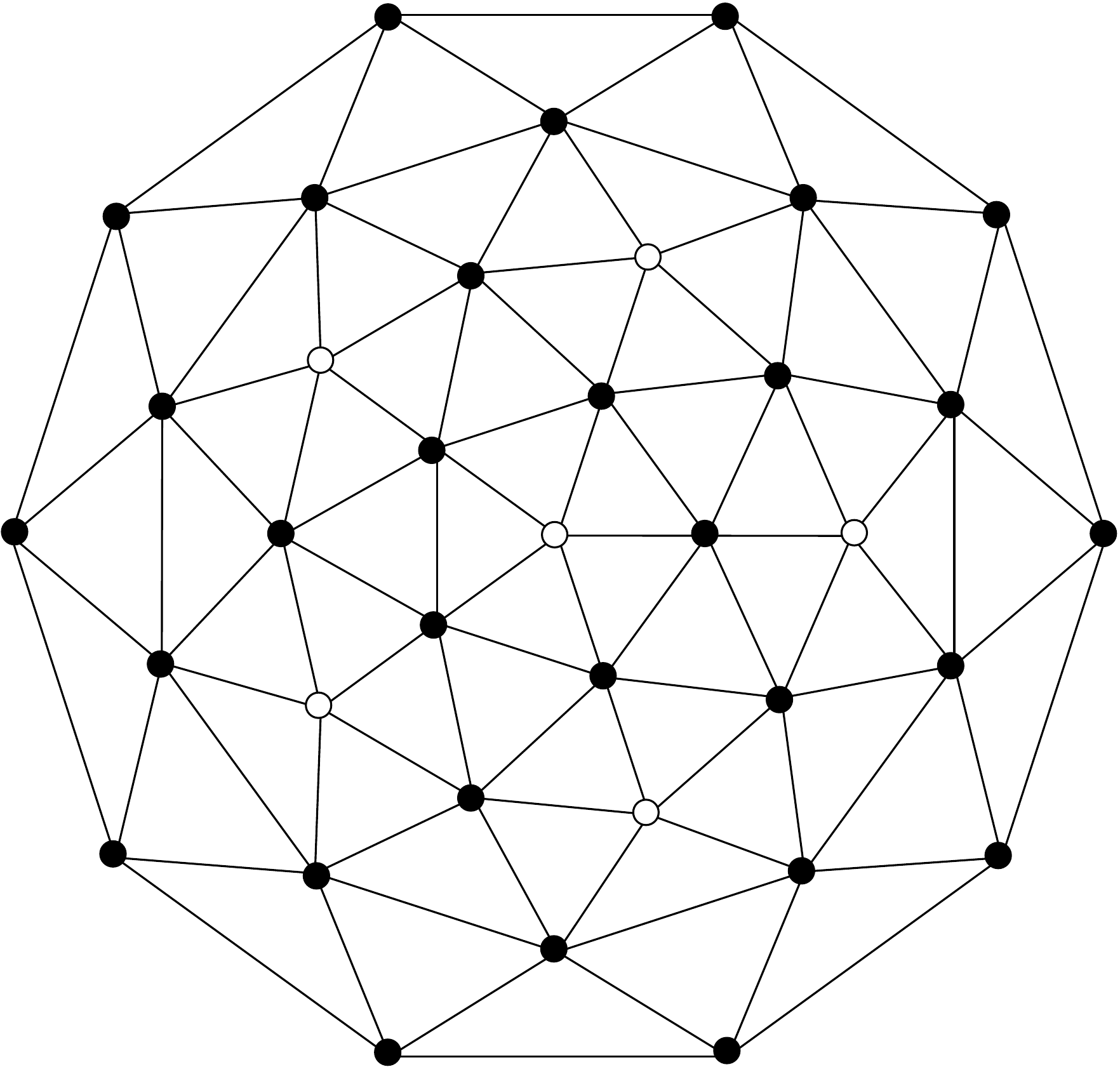}
    }    
    \caption{Four irreducible 6-clusters. 
    }
    \label{fig:4_irred_6_clusters}
\end{figure}

By checking all possible reductions which involve a 5-vertex which is part of one of these four clusters, it can be seen that dual IPR fullerenes which contain one of these clusters do not have a reduction to a smaller dual IPR fullerene where at least one of the 5-vertices involved in the reduction is in one of these four clusters. We call clusters with this property \textit{globally irreducible}\index{cluster!globally irreducible}. This gives us:

\begin{corollary} 
Every dual IPR fullerene which contains two 6-clusters $c$ and $d$ with
$c, d \in \{I, II,\allowbreak III, IV\}$ is \underline{not} reducible to a smaller dual IPR fullerene.
\end{corollary}

Also note that dual caps which contain a connected subgraph of six 5-vertices which is isomorphic to a subgraph of a cluster $c \in \{I, II, III, IV\}$ have different boundary parameters for each different $c$. Therefore dual IPR fullerenes which contain two 6-clusters $c$ and $d$ with $c \in \{I, II, III, IV\}$ and $d \in \{I, II, III, IV\} \setminus \{c\}$ do not exist.

All dual caps which contain a connected subgraph with six 5-vertices which is isomorphic to a subgraph of cluster I-IV are globally irreducible as well. So all IPR fullerenes which can be decomposed into 2 caps where both caps are globally irreducible are not reducible to a smaller IPR fullerene.

By using the generator for caps from Brinkmann et al.~\cite{brinkmann_02}, we were able to determine that all dual IPR caps with boundary parameters (5,5) (respectively (8,2) and (9,0)) contain a connected subgraph with six 5-vertices which is isomorphic to a subgraph of cluster I (respectively II and III). However there are caps with boundary parameters (10,0) which do not contain a connected subgraph with six 5-vertices which is isomorphic to a subgraph of cluster IV. This gives us the following corollary:

\begin{corollary} \label{cor:6-cluster_3}
Every IPR fullerene which contains a (5,5), (8,2) or (9,0) boundary is \underline{not} reducible to a smaller IPR fullerene.
\end{corollary}

We will now show that all dual IPR fullerenes which have a (10,0) boundary are reducible, except for dual fullerenes where both caps contain a connected subgraph with six 5-vertices which is isomorphic to a subgraph of cluster IV and for a limited number of dual fullerenes which contain an irreducible 12-cluster.

By using the modified version of the generator for caps from Brinkmann et al.~\cite{brinkmann_02}, we were able to generate all IPR (10,0) kernels. The largest one has 60 vertices, so an IPR fullerene which contains a (10,0) cap and has no reducible (10,0) hexagon rings has at most $2 \cdot 60 + 2 \cdot 10 = 140$ vertices. Using \textit{fullgen} we determined all of these fullerenes. These results were also independently confirmed by \textit{buckygen}.

All of these dual IPR fullerenes are reducible, except the ones where both dual caps contain a connected subgraph with six 5-vertices isomorphic to a subgraph of cluster IV and a limited number of dual fullerenes which contain a 12-cluster. In Section~\ref{section:irred_7_12_clusters} we will show which dual fullerenes containing a 12-cluster are irreducible. 

We verified that for each of these reducible IPR fullerenes $F$ there is an $r$ such that the fullerenes obtained by adding $r$ $(10,0)$ rings of hexagons to $F$ have a reduction which is entirely within one cap. 
We also verified that all fullerenes obtained from $F$ by adding less than $r$ $(10,0)$ rings of hexagons are reducible as well. The irreducible dual IPR fullerenes which contain a 12-cluster where the dual caps do not contain a connected subgraph with six 5-vertices which is isomorphic to a subgraph of cluster IV also become reducible if a $(10,0)$ ring of 6-vertices is added. Also for these dual fullerenes there is an $r$ such that the dual fullerenes obtained by adding $r$ $(10,0)$ rings of 6-vertices have a reduction which is entirely within one dual cap (and all of these dual fullerenes obtained by adding less than $r$ $(10,0)$ rings of 6-vertices are reducible as well).

This gives us the following corollary:

\begin{corollary} \label{cor:10_0_red}
Every dual IPR fullerene which contains a (10,0) boundary is reducible to a smaller dual IPR fullerene, except for dual fullerenes where both dual caps contain a connected subgraph with six 5-vertices which is isomorphic to a subgraph of cluster IV, and for a limited number of dual fullerenes which contain an irreducible 12-cluster.
\end{corollary}

Together with the other corollaries from this section, this gives us: 

\begin{corollary} \label{lemma:6_cluster}
All dual IPR fullerenes which contain a 6-cluster are reducible to a smaller dual IPR fullerene, unless the dual fullerene contains 2 clusters $c$ with $c \in \{I, II, III, IV\}$
\end{corollary}

\subsection{Reducibility of $k$-clusters $(7 \le k \le 12)$} \label{section:irred_7_12_clusters}

Now we will prove that all dual IPR fullerenes which contain a $k$-cluster with $7 \le k \le 11$ are reducible to a smaller dual IPR fullerene. We will also prove that there are only a limited number of dual fullerenes which contain a 12-cluster which are not reducible to a smaller dual IPR fullerene and determine them.

For a given patch with $k$ pentagons $(7 \le k \le 12)$, we can compute an upper bound for the number of vertices of a fullerene which contains this patch by using the results from~\cite{bornhoft_03}. Suppose for example that we have a patch $P$ with 7 pentagons, $h_P$ hexagons and boundary length $l$. We can determine an upper bound for the number of hexagons $h$ in a patch with the same boundary length and 5 pentagons by using Theorem 12 of~\cite{bornhoft_03} as follows:

\begin{eqnarray*} 
\frac{l + 1}{2}  & \ge &    \biggl\lceil \sqrt{2h + \frac{113}{4}} + \frac{1}{2} \biggr\rceil \\
\frac{l}{2}        & \ge &  \sqrt{2h + \frac{113}{4}} \\
h                     & \le &  \frac{l^2 - 113}{8}
\end{eqnarray*}

So the number of faces in a fullerene containing $P$ is at most $7 + h_P + 5 + \frac{l^2 - 113}{8}$. For patches with $k$ $(8 \le k \le 12)$ pentagons, an upper bound for the number of faces of a fullerene which contains such a patch is obtained in a similar way. Based on this, we computed an upper bound for the number of vertices of a fullerene containing the dual of a $k$-cluster $(7 \le k \le 12)$ (see~\cite{thesis_jan} for details). The results are shown in Table~\ref{table:upper_bound_cluster}. 

Note that these upper bounds are very coarse since the patches with the largest number of hexagons given in~\cite{bornhoft_03} for a given number of pentagons and boundary length are not IPR if the patch contains at least 2 pentagons.

\begin{table}
\begin{center}
\begin{tabular}{| c | c |}
\hline
$k$	& max nv \\
\hline
7    & 462 \\
8    & 330 \\
9    & 296 \\
10  & 286 \\
11  & 286 \\
12  & 292 \\
\hline
\end{tabular}

\caption{Upper bound for the number of vertices of a fullerene containing the dual of a $k$-cluster.}
\label{table:upper_bound_cluster}

\end{center}
\end{table}

Using \textit{fullgen} we generated all IPR fullerenes up to 330 vertices and tested them for reducibility. This was independently verified by \textit{buckygen}. We obtained the following results:

\begin{observation} \label{lemma:8_11_cluster}
All dual IPR fullerenes which contain a $k$-cluster $(8 \le k \le 11)$ are reducible to a smaller dual IPR fullerene.
\end{observation}

\begin{observation} \label{lemma:12_cluster}
There are exactly 56 irreducible dual IPR fullerenes which contain a 12-cluster. The largest one has 58 vertices or $2 \cdot (58 - 2) = 112$ faces.
\end{observation}

\begin{observation} \label{lemma:12_clusterb}
There are exactly 36 irreducible dual IPR fullerenes which contain a 12-cluster and which do not have a dual cap which contains a connected subgraph with six 5-vertices which is isomorphic to a subgraph of cluster I, II, III or IV.
\end{observation}

It was not feasible to generate all IPR fullerenes up to 462 vertices with \textit{fullgen}. However, our generator for locally irreducible clusters was still fast enough to generate all locally irreducible 7-clusters. 
By using these specific 7-clusters $C$ which have boundary length $b_C$ in the formula $|V(C)| + 5 + \frac{b_{C}^2 - 113}{8}$ (where $|V(C)|$ stands for the number of vertices of $C$), we were able to determine that fullerenes which contain the dual of one of these locally irreducible 7-clusters have at most 166 vertices.
Using \textit{fullgen} we generated all these fullerenes and tested them for reducibility. We obtained the following result (which was independently confirmed by \textit{buckygen}):

\begin{corollary} \label{lemma:7_cluster}
All dual IPR fullerenes which contain a 7-cluster are reducible to a smaller dual IPR fullerene.
\end{corollary}

Actually we only had to prove that dual IPR fullerenes which contain one 7-cluster and five 1-clusters (or one 8-cluster and four 1-clusters etc.) are reducible. Since e.g.\ a dual fullerene consisting of a 7-cluster and a 5-cluster is always reducible since all 5-clusters are locally reducible (see Observation~\ref{lemma:235_cluster}).

By noting that in a dual IPR fullerene every 5-vertex is part of a cluster, together Corollaries~\ref{lemma:1_cluster}, \ref{lemma:235_cluster}, \ref{lemma:4_cluster}, \ref{lemma:6_cluster}, \ref{lemma:8_11_cluster}, \ref{lemma:12_clusterb} and \ref{lemma:7_cluster} lead to the following theorem:

\begin{theorem}
The class of irreducible dual IPR fullerenes consists of 4 infinite families of dual IPR fullerenes which contain two 6-clusters $c$ with $c \in \{I, II, III, IV\}$ and 36 dual IPR fullerenes which contain a 12-cluster.
\end{theorem}

\subsection{Open questions}
When classifying these irreducible IPR fullerenes we encountered some open questions. Future work might include solving these open questions:

\begin{itemize}
\item Can every fullerene be split into two \textit{caps}? By performing a computer search, we verified that all fullerenes up to 200 vertices can be split into two caps.

\item Does a 12-cluster uniquely determine a dual fullerene? Or equivalently: does a boundary sequence uniquely describe the interior of a subpatch of a fullerene which only consists of hexagons? 

It is known~\cite{brinkmann_05,guo_02} that the boundary of a hexagon patch determines the number of faces of the patch. It is also known that the boundary sequence uniquely describes the interior of a hexagonal patch if it is a subgraph of the hexagonal lattice and it has been shown by Guo et al.~\cite{guo_02} that this is not the case if the patch is not necessarily a subgraph of the hexagon lattice.
For hexagon patches which are subgraphs of fullerenes, it is unknown.
\end{itemize}

\section{Generation algorithm}
\label{section:generation_algorithm}
In order to generate all IPR fullerenes with $n$ vertices, the generation algorithm recursively applies the IPR construction operations from Section~\ref{section:construction_operations} to all irreducible IPR fullerenes with at most $n$ vertices.

The 4 infinity families of irreducible dual IPR nanotube fullerenes which contain two 6-clusters $c$ with $c \in \{I, II, III, IV\}$ consist of dual caps with boundary parameters $(5,5)$, $(8,2)$, $(9,0)$ or $(10,0)$, respectively. They are generated by adding rings of 6-vertices with the respective parameters in all possible ways. Since there are only a small number of irreducible IPR fullerenes (see Section~\ref{section:results}), we use the following simple method to make sure no isomorphic irreducible IPR fullerenes are output: we compute and store a canonical form for each generated irreducible IPR nanotube fullerene and only output the irreducible fullerenes which were not generated before. For details about the canonical form, we refer to~\cite{brinkmann_07}.

To make sure that no isomorphic reducible IPR fullerenes are output, we use the canonical construction path method~\cite{mckay_98}. The isomorphism rejection method is very similar to the method used in~\cite{fuller-paper} and therefore we refer to that article for more details and a proof that exactly one representative of each isomorphism class of dual IPR fullerenes is output.

\section{Testing and results}
\label{section:results}

We implemented our algorithm for the recursive generation of IPR fullerenes and incorporated it in the program \textit{buckygen}~\cite{fuller-paper} which can be downloaded from~\cite{buckygen-site}. \textit{Buckygen} is also part of the \textit{CaGe} software package~\cite{cage}. \textit{Buckygen} can be used to recursively generate IPR fullerenes by executing it with the command line argument \verb|-I|. We will refer to this program as \textit{buckygen IPR}.

\textit{Buckygen} can also be used to generate IPR fullerenes by generating all fullerenes and using a filter and look-aheads for IPR fullerenes. We will refer to this generator as \textit{buckygen IPR filter}.

A comparison of the running times for generating IPR fullerenes is given in Table~\ref{table:fullerene_times_ipr}. The programs were compiled with
gcc and executed on an Intel Xeon L5520 CPU at 2.27 GHz. The running
times include writing the IPR fullerenes to a null device.

As can be seen from that table, \textit{buckygen IPR} is significantly faster than \textit{fullgen}~\cite{brinkmann_97}. \textit{Buckygen} constructs larger fullerenes from smaller ones. So generating all IPR fullerenes
with at most $n$ vertices gives only a small overhead compared to
generating all IPR fullerenes with exactly $n$ vertices. In \textit{fullgen} the
overhead is considerably larger as it does not construct fullerenes from
smaller fullerenes. 

The speedup of \textit{buckygen IPR} compared to \textit{buckygen IPR filter} is decreasing because in \textit{buckygen IPR filter} several lemmas can be applied which allow to determine a good bound on the length of the shortest reduction (see~\cite{fuller-paper}), while these cannot be applied to \textit{buckygen IPR}. Furthermore the ratio of IPR fullerenes among all fullerenes is increasing, thus the ratio of fullerenes which are rejected by \textit{buckygen IPR filter} because they are not IPR is decreasing. However for the fullerene sizes which are important for practical purposes, \textit{buckygen IPR} is significantly faster than other generators for IPR fullerenes.

\begin{table}
\small
\centering
\begin{tabular}{| c || r | c | c | c | c |}
\hline 
number of & time (s) & fullerenes/s & fg IPR (s) / & fg IPR (s) / & bg IPR filter (s) / \\
vertices & (bg IPR) & (bg IPR) & bg IPR (s) & bg IPR filter (s) & bg IPR (s)\\
\hline
200  &  4 110  &  3 809  &  1.88  &  0.80 & 2.34\\
230  &  22 481  &  3 836  &  2.14  &  0.96 & 2.23\\
260  &  104 831  &  3 456  &  2.18  &  1.03 & 2.21\\
280  &  274 748  &  3 066  &  2.19  &  1.10 & 2.00\\
300  &  678 331  &  2 686  &  2.19  &  1.16 & 1.88\\
320  &  1 591 041  &  2 329  &  1.99  &  1.14 & 1.75\\
340  &  3 613 915  &  1 981  &  1.73  &  1.09 & 1.60\\
360  &  8 135 063  &  1 625  &  1.51  &  1.05 & 1.43\\
\hline
0--140  &  17.5  &  33 055  &  19.62  &  1.99  &  9.85\\
200--250 &  79 152  &  28 321  &  14.37  &  6.66  &  2.16\\
290--300  &  776 910  &  11 753  &  7.83  &  4.11  &  1.91\\
\hline
\end{tabular}

\caption{Running times and generation rates for IPR fullerenes. \textit{Bg} stands for \textit{buckygen} and \textit{fg} stands for \textit{fullgen}.}

\label{table:fullerene_times_ipr}
\end{table}

We used \textit{buckygen IPR} to generate all IPR fullerenes up to 400 vertices. These results were independently confirmed by \textit{buckygen IPR filter} and \textit{fullgen IPR} up to 380 vertices.

The counts of all fullerenes, irreducible IPR fullerenes and IPR fullerenes up to 400 vertices can be found in Tables~\ref{table:fuller_counts_1}-\ref{table:fuller_counts_3}. Some of these graphs can be downloaded from the \textit{House of Graphs}~\cite{hog} at \url{http://hog.grinvin.org/Fullerenes}

\begin{table}
\centering
{\small 
\begin{tabular}{| c | c | c | c | c |}
\hline 
 nv & nf & fullerenes & irred.\, IPR fullerenes & IPR fullerenes\\
\hline 
20  &  12  &  1  &  0  &  0  \\
22  &  13  &  0  &  0  &  0  \\
24  &  14  &  1  &  0  &  0  \\
26  &  15  &  1  &  0  &  0  \\
28  &  16  &  2  &  0  &  0  \\
30  &  17  &  3  &  0  &  0  \\
32  &  18  &  6  &  0  &  0  \\
34  &  19  &  6  &  0  &  0  \\
36  &  20  &  15  &  0  &  0  \\
38  &  21  &  17  &  0  &  0  \\
40  &  22  &  40  &  0  &  0  \\
42  &  23  &  45  &  0  &  0  \\
44  &  24  &  89  &  0  &  0  \\
46  &  25  &  116  &  0  &  0  \\
48  &  26  &  199  &  0  &  0  \\
50  &  27  &  271  &  0  &  0  \\
52  &  28  &  437  &  0  &  0  \\
54  &  29  &  580  &  0  &  0  \\
56  &  30  &  924  &  0  &  0  \\
58  &  31  &    1 205    &  0  &  0  \\
60  &  32  &    1 812    &  1  &  1  \\
62  &  33  &    2 385    &  0  &  0  \\
64  &  34  &    3 465    &  0  &  0  \\
66  &  35  &    4 478    &  0  &  0  \\
68  &  36  &    6 332    &  0  &  0  \\
70  &  37  &    8 149    &  1  &  1  \\
72  &  38  &    11 190    &  1  &  1  \\
74  &  39  &    14 246    &  1  &  1  \\
76  &  40  &    19 151    &  2  &  2  \\
78  &  41  &    24 109    &  4  &  5  \\
80  &  42  &    31 924    &  7  &  7  \\
82  &  43  &    39 718    &  8  &  9  \\
84  &  44  &    51 592    &  11  &  24  \\
86  &  45  &    63 761    &  1  &  19  \\
88  &  46  &    81 738    &  3  &  35  \\
90  &  47  &    99 918    &  2  &  46  \\
92  &  48  &    126 409    &  3  &  86  \\
94  &  49  &    153 493    &  0  &  134  \\
96  &  50  &    191 839    &  4  &  187  \\
98  &  51  &    231 017    &  1  &  259  \\
100  &  52  &    285 914    &  3  &  450  \\
102  &  53  &    341 658    &  0  &  616  \\
104  &  54  &    419 013    &  1  &  823  \\
106  &  55  &    497 529    &  0  &    1 233    \\
108  &  56  &    604 217    &  2  &    1 799    \\
110  &  57  &    713 319    &  1  &    2 355    \\
112  &  58  &    860 161    &  2  &    3 342    \\
114  &  59  &    1 008 444    &  2  &    4 468    \\
116  &  60  &    1 207 119    &  1  &    6 063    \\
118  &  61  &    1 408 553    &  0  &    8 148    \\
120  &  62  &    1 674 171    &  4  &    10 774    \\
122  &  63  &    1 942 929    &  0  &    13 977    \\
124  &  64  &    2 295 721    &  1  &    18 769    \\
126  &  65  &    2 650 866    &  0  &    23 589    \\
128  &  66  &    3 114 236    &  1  &    30 683    \\
130  &  67  &    3 580 637    &  1  &    39 393    \\
132  &  68  &    4 182 071    &  3  &    49 878    \\
134  &  69  &    4 787 715    &  0  &    62 372    \\
136  &  70  &    5 566 949    &  1  &    79 362    \\
138  &  71  &    6 344 698    &  0  &    98 541    \\
140  &  72  &    7 341 204    &  3  &    121 354    \\
142  &  73  &    8 339 033    &  0  &    151 201    \\
144  &  74  &    9 604 411    &  1  &    186 611    \\
146  &  75  &    10 867 631    &  0  &    225 245    \\
\hline
\end{tabular}
}
\caption{Counts of all fullerenes, irreducible IPR fullerenes and IPR fullerenes. nv is the number of vertices and nf is the number of faces.}

\label{table:fuller_counts_1}

\end{table}

\begin{table}
\centering
{\small 
\begin{tabular}{| c | c | c | c | c |}
\hline 
 nv & nf & fullerenes & irred.\, IPR fullerenes & IPR fullerenes\\
\hline 
148  &  76  &    12 469 092    &  1  &    277 930\\
150  &  77  &    14 059 174    &  3  &    335 569\\
152  &  78  &    16 066 025    &  1  &    404 667\\
154  &  79  &    18 060 979    &  0  &    489 646\\
156  &  80  &    20 558 767    &  1  &    586 264\\
158  &  81  &    23 037 594    &  0  &    697 720\\
160  &  82  &    26 142 839    &  4  &    836 497\\
162  &  83  &    29 202 543    &  0  &    989 495\\
164  &  84  &    33 022 573    &  1  &    1 170 157\\
166  &  85  &    36 798 433    &  0  &    1 382 953\\
168  &  86  &    41 478 344    &  3  &    1 628 029\\
170  &  87  &    46 088 157    &  1  &    1 902 265\\
172  &  88  &    51 809 031    &  1  &    2 234 133\\
174  &  89  &    57 417 264    &  0  &    2 601 868\\
176  &  90  &    64 353 269    &  1  &    3 024 383\\
178  &  91  &    71 163 452    &  0  &    3 516 365\\
180  &  92  &    79 538 751    &  3  &    4 071 832\\
182  &  93  &    87 738 311    &  0  &    4 690 880\\
184  &  94  &    97 841 183    &  1  &    5 424 777\\
186  &  95  &    107 679 717    &  2  &    6 229 550\\
188  &  96  &    119 761 075    &  1  &    7 144 091\\
190  &  97  &    131 561 744    &  1  &    8 187 581\\
192  &  98  &    145 976 674    &  1  &    9 364 975\\
194  &  99  &    159 999 462    &  0  &    10 659 863\\
196  &  100  &    177 175 687    &  1  &    12 163 298\\
198  &  101  &    193 814 658    &  0  &    13 809 901\\
200  &  102  &    214 127 742    &  4  &    15 655 672\\
202  &  103  &    233 846 463    &  0  &    17 749 388\\
204  &  104  &    257 815 889    &  3  &    20 070 486\\
206  &  105  &    281 006 325    &  0  &    22 606 939\\
208  &  106  &    309 273 526    &  1  &    25 536 557\\
210  &  107  &    336 500 830    &  1  &    28 700 677\\
212  &  108  &    369 580 714    &  1  &    32 230 861\\
214  &  109  &    401 535 955    &  0  &    36 173 081\\
216  &  110  &    440 216 206    &  1  &    40 536 922\\
218  &  111  &    477 420 176    &  0  &    45 278 722\\
220  &  112  &    522 599 564    &  3  &    50 651 799\\
222  &  113  &    565 900 181    &  2  &    56 463 948\\
224  &  114  &    618 309 598    &  1  &    62 887 775\\
226  &  115  &    668 662 698    &  0  &    69 995 887\\
228  &  116  &    729 414 880    &  1  &    77 831 323\\
230  &  117  &    787 556 069    &  1  &    86 238 206\\
232  &  118  &    857 934 016    &  1  &    95 758 929\\
234  &  119  &    925 042 498    &  0  &    105 965 373\\
236  &  120  &    1 006 016 526    &  1  &    117 166 528\\
238  &  121  &    1 083 451 816    &  0  &    129 476 607\\
240  &  122  &    1 176 632 247    &  6  &    142 960 479\\
242  &  123  &    1 265 323 971    &  0  &    157 402 781\\
244  &  124  &    1 372 440 782    &  1  &    173 577 766\\
246  &  125  &    1 474 111 053    &  0  &    190 809 628\\
248  &  126  &    1 596 482 232    &  1  &    209 715 141\\
250  &  127  &    1 712 934 069    &  1  &    230 272 559\\
252  &  128  &    1 852 762 875    &  1  &    252 745 513\\
254  &  129  &    1 985 250 572    &  0  &    276 599 787\\
256  &  130  &    2 144 943 655    &  1  &    303 235 792\\
258  &  131  &    2 295 793 276    &  2  &    331 516 984\\
260  &  132  &    2 477 017 558    &  3  &    362 302 637\\
262  &  133  &    2 648 697 036    &  0  &    395 600 325\\
264  &  134  &    2 854 536 850    &  1  &    431 894 257\\
266  &  135  &    3 048 609 900    &  0  &    470 256 444\\
268  &  136  &    3 282 202 941    &  1  &    512 858 451\\
270  &  137  &    3 501 931 260    &  1  &    557 745 670\\
272  &  138  &    3 765 465 341    &  1  &    606 668 511\\
274  &  139  &    4 014 007 928    &  0  &    659 140 287\\
\hline
\end{tabular}
}
\caption{Counts of all fullerenes, irreducible IPR fullerenes and IPR fullerenes (continued). nv is the number of vertices and nf is the number of faces.}

\label{table:fuller_counts_2}

\end{table}

\begin{table}
\centering
{\small 
\begin{tabular}{| c | c | c | c | c |}
\hline 
 nv & nf & fullerenes & irred.\, IPR fullerenes & IPR fullerenes\\
\hline 
276  &  140  &    4 311 652 376    &  3  &    716 217 922\\
278  &  141  &    4 591 045 471    &  0  &    776 165 188\\
280  &  142  &    4 926 987 377    &  4  &    842 498 881\\
282  &  143  &    5 241 548 270    &  0  &    912 274 540\\
284  &  144  &    5 618 445 787    &  1  &    987 874 095\\
286  &  145  &    5 972 426 835    &  0  &    1 068 507 788\\
288  &  146  &    6 395 981 131    &  1  &    1 156 161 307\\
290  &  147  &    6 791 769 082    &  1  &    1 247 686 189\\
292  &  148  &    7 267 283 603    &  1  &    1 348 832 364\\
294  &  149  &    7 710 782 991    &  2  &    1 454 359 806\\
296  &  150  &    8 241 719 706    &  1  &    1 568 768 524\\
298  &  151  &    8 738 236 515    &  0  &    1 690 214 836\\
300  &  152  &    9 332 065 811    &  3  &    1 821 766 896\\
302  &  153  &    9 884 604 767    &  0  &    1 958 581 588\\
304  &  154  &    10 548 218 751    &  1  &    2 109 271 290\\
306  &  155  &    11 164 542 762    &  0  &    2 266 138 871\\
308  &  156  &    11 902 015 724    &  1  &    2 435 848 971\\
310  &  157  &    12 588 998 862    &  1  &    2 614 544 391\\
312  &  158  &    13 410 330 482    &  3  &    2 808 510 141\\
314  &  159  &    14 171 344 797    &  0  &    3 009 120 113\\
316  &  160  &    15 085 164 571    &  1  &    3 229 731 630\\
318  &  161  &    15 930 619 304    &  0  &    3 458 148 016\\
320  &  162  &    16 942 010 457    &  4  &    3 704 939 275\\
322  &  163  &    17 880 232 383    &  0  &    3 964 153 268\\
324  &  164  &    19 002 055 537    &  1  &    4 244 706 701\\
326  &  165  &    20 037 346 408    &  0  &    4 533 465 777\\
328  &  166  &    21 280 571 390    &  1  &    4 850 870 260\\
330  &  167  &    22 426 253 115    &  3  &    5 178 120 469\\
332  &  168  &    23 796 620 378    &  1  &    5 531 727 283\\
334  &  169  &    25 063 227 406    &  0  &    5 900 369 830\\
336  &  170  &    26 577 912 084    &  1  &    6 299 880 577\\
338  &  171  &    27 970 034 826    &  0  &    6 709 574 675\\
340  &  172  &    29 642 262 229    &  3  &    7 158 963 073\\
342  &  173  &    31 177 474 996    &  0  &    7 620 446 934\\
344  &  174  &    33 014 225 318    &  1  &    8 118 481 242\\
346  &  175  &    34 705 254 287    &  0  &    8 636 262 789\\
348  &  176  &    36 728 266 430    &  3  &    9 196 920 285\\
350  &  177  &    38 580 626 759    &  1  &    9 768 511 147\\
352  &  178  &    40 806 395 661    &  1  &    10 396 040 696\\
354  &  179  &    42 842 199 753    &  0  &    11 037 658 075\\
356  &  180  &    45 278 616 586    &  1  &    11 730 538 496\\
358  &  181  &    47 513 679 057    &  0  &    12 446 446 419\\
360  &  182  &    50 189 039 868    &  4  &    13 221 751 502\\
362  &  183  &    52 628 839 448    &  0  &    14 010 515 381\\
364  &  184  &    55 562 506 886    &  1  &    14 874 753 568\\
366  &  185  &    58 236 270 451    &  2  &    15 754 940 959\\
368  &  186  &    61 437 700 788    &  1  &    16 705 334 454\\
370  &  187  &    64 363 670 678    &  1  &    17 683 643 273\\
372  &  188  &    67 868 149 215    &  1  &    18 744 292 915\\
374  &  189  &    71 052 718 441    &  0  &    19 816 289 281\\
376  &  190  &    74 884 539 987    &  1  &    20 992 425 825\\
378  &  191  &    78 364 039 771    &  0  &    22 186 413 139\\
380  &  192  &    82 532 990 559    &  3  &    23 475 079 272\\
382  &  193  &    86 329 680 991    &  0  &    24 795 898 388\\
384  &  194  &    90 881 152 117    &  3  &    26 227 197 453\\
386  &  195  &    95 001 297 565    &  0  &    27 670 862 550\\
388  &  196  &    99 963 147 805    &  1  &    29 254 036 711\\
390  &  197  &    104 453 597 992    &  1  &    30 852 950 986\\
392  &  198  &    109 837 310 021    &  1  &    32 581 366 295\\
394  &  199  &    114 722 988 623    &  0  &    34 345 173 894\\
396  &  200  &    120 585 261 143    &  1  &    36 259 212 641\\
398  &  201  &    125 873 325 588    &  0  &    38 179 777 473\\
400  &  202  &    132 247 999 328    &  4  &    40 286 153 024\\
\hline
\end{tabular}
}
\caption{Counts of all fullerenes, irreducible IPR fullerenes and IPR fullerenes (continued). nv is the number of vertices and nf is the number of faces.}

\label{table:fuller_counts_3}

\end{table}

\begin{acknowledgements}
Most computations for this work were carried out using the Stevin Supercomputer Infrastructure at Ghent University.
We would like to thank Gunnar Brinkmann and Jack Graver for useful suggestions.
\end{acknowledgements}

\bibliographystyle{plain}
\bibliography{references}

\end{document}